\documentclass[sn-mathphys, pdflatex]{sn-jnl}

\usepackage{color}
\usepackage{cleveref}
\usepackage{enumitem}
\usepackage{mathrsfs}
\usepackage{amsmath,amssymb,amsthm,amsfonts,mathtools,amscd}
\usepackage{bm}
\usepackage{graphicx}
\usepackage{enumerate}
\usepackage{url}
\usepackage{color}
\usepackage{enumitem}
\usepackage{bbding}
\usepackage{geometry}
\usepackage{txfonts}
\usepackage{xcolor}
\usepackage[numbers,sort&compress]{natbib}
\bibpunct[, ]{[}{]}{,}{n}{,}{,}
\newcommand{\Sph}{\mathbb{S}}
\DeclareMathOperator*\bdry{bdry}
\DeclareMathOperator*\dom{dom}

\DeclareMathOperator*\pos{pos}
\DeclareMathOperator*\gph{gph}
\DeclareMathOperator*\inte{int}

\DeclareMathOperator*\cl{cl}
\DeclareMathOperator*\proj{proj}

\DeclareMathOperator*\rg{rg}
\DeclareMathOperator*\lip{lip}
\DeclareMathOperator*\glimsup{g-limsup}

\jyear{2022}%

\theoremstyle{thmstyleone}%
\newtheorem{theorem}{Theorem}[section]
\newtheorem{lemma}[theorem]{Lemma}
\newtheorem{corollary}[theorem]{Corollary}
\newtheorem{proposition}[theorem]{Proposition}

\theoremstyle{thmstyletwo}%
\newtheorem{example}{Example}[section]%

\theoremstyle{thmstylethree}%
\newtheorem{definition}{Definition}[section]%

\begin{document}

\title[Projectional Coderivatives and Calculus Rules]{Projectional Coderivatives and Calculus Rules}

\author*[1]{\fnm{Wenfang} \sur{Yao}}\email{dorothyywf@gmail.com}

\author[2]{\fnm{Kaiwen} \sur{Meng}}\email{mengkw@swufe.edu.cn}

\author[3]{\fnm{Minghua} \sur{Li}}\email{minghuali20021848@163.com}

\author[1]{\fnm{Xiaoqi} \sur{Yang}}\email{mayangxq@polyu.edu.hk}

\affil*[1]{\orgdiv{Department of Applied Mathematics}, \orgname{The Hong Kong Polytechnic University}, \orgaddress{\street{Kowloon}, \city{Hong Kong}, \country{China}}}

\affil[2]{\orgdiv{School of Mathematics}, \orgname{Southwestern University of Finance and Economics}, \orgaddress{ \city{Chengdu}, \postcode{611130},  \country{China}}}

\affil[3]{\orgdiv{School of Mathematics and Big Data}, \orgname{Chongqing University of Arts and Sciences}, \orgaddress{\street{Yongchuan}, \city{Chongqing}, \postcode{402160}, \country{China}}}

\abstract{This paper is devoted to the study of a newly introduced tool, projectional coderivatives and the corresponding calculus rules in finite dimensional spaces. We show that when the restricted set has some nice properties, more specifically, it is a smooth manifold, the projectional coderivative can be refined as a fixed-point expression. We will also improve the generalized Mordukhovich criterion to give a complete characterization of the relative Lipschitz-like property under such a setting. Chain rules and sum rules are obtained to facilitate the application of the tool to a wider range of parametric problems. }

\keywords{Projectional coderivatives, calculus rules, generalized Mordukhovich criterion, relative Lipschitz-like property}

\pacs[MSC Classification]{49J53, 49K40, 58C07, 90C31}

\maketitle

\section{Introduction}\label{sect:intro}
Among various stability properties, the Lipschitz-like property (also known as the Aubin property) plays a central role and has deep consequences with other stability properties like error bounds, metric regularity and calmness, which are widely adopted in convergence analysis of iterative algorithms. 
It can be traced back to \cite{aubin1984lipschitz} and inherits the name to reveal its intrinsic nature as a local Lipschitzian behavior of multifunctions (\cite{mordukhovich2018variational}).
By virtue of the Mordukhovich criterion introduced in  \cite{mordukhovich1992sensitivity}, it can be fully captured by coderivatives, along with the graphical modulus described by the outer norm of coderivatives. This criterion provides a geometric perspective on local stability behaviors of the graph of the set-valued mapping.
The calculus rules of coderivatives widely adopted were largely initiated in \cite{mordukhovich1994generalized} and introduced in later monographs \cite{VaAn} and \cite{mordukhovich2006variational} under the assumptions of local boundedness and graph-convexity, in finite as well as infinite dimensions. These rules facilitate analyzing the Lipschitz-like property of broader types of problems. Abundant calculus fitting diverse parametric systems can also be found in \cite{mordukhovich2007coderivative, LevyMord2004, Huyen2016, lee2005quadratic}. For more introduction on parametric optimization problems, see monographs by \cite{bonnans2013perturbation,dontchev2009implicit,ioffe2017variational} and \cite{klatte2006nonsmooth}.

However, the Lipschitz-like property has an implicit prerequisite that the reference point lies in the interior of the domain, which can be observed via both the definition of the property and the local boundedness requirement of coderivatives in the Mordukhovich criterion. Besides, in practical problems, perturbations often come in structures or in specific directions. Thus over the years the relative stability has gained much attention. Most of these research employed the tool directional limiting coderivatives, introduced in \cite{gfrerer2013directional} with directional limiting calculus initiated in \cite{ginchev2011directionally}. 
For directionally Lipschitzian single-valued mappings and generalized directional derivatives, see \cite{clarke1990optimization}. In \cite{GfrOut2016} sufficient conditions were established for the calmness and the Lipschitz-like property of implicit multifunctions by using a directional limiting coderivative along with graphical derivative. For other stability properties relative to a set,  see \cite{bonnans2013perturbation} and \cite{van2015directional,ioffe2010regularity,arutyunov2006directional} for the relative metric regularity, \cite{mordukhovich2004restrictive} for the restrictive metric regularity and \cite{Benko2020} for the relative isolated calmness. 

Recently, sufficient conditions for the Lipschitz-like property relative to a closed set of the solution map for a class of parameterized variational systems were derived in \cite{Benko2020}. These conditions require computation of directional limiting coderivatives of the normal-cone mapping for the critical directions. 
However, they only provided sufficiency which failed to characterize the property fully, especially when the candidate point lies on the boundary. In \cite{Meng2020}, a new tool, the projectional coderivative,  was introduced and both sufficiency and necessity were provided for characterizing the Lipschitz-like property relative to a closed and convex set. This complete characterization is called the generalized Mordukhovich criterion, as when the candidate point lies in the interior of the relative set, it reduces to the Mordukhovich criterion. This verifiable condition also pins down the associated graphical modulus of the set-valued mapping relative to the closed and convex set with the outer norm of projectional coderivatives. The comparison between the tool in \cite{Meng2020} and that in \cite{Benko2020} was also demonstrated via examples both in \cite{Meng2020, yao2022relative}. 

For this newly acquainted tool, only a few properties and examples are presented due to its complicated nature. It involves interactions between normal cone of the set-valued mapping and projection onto the tangent cone of the set in the neighborhood. Carrying the projection into the outer limit brings difficulties to generalization of applicability of this tool in two ways: one is that the projected normal cone does not enjoy outer-semicontinuity as normal cone does. The other one is that the projection creates asymmetrical effects: the element corresponding to the domain part takes the projection while the element corresponding to the range part does not. In \cite{yao2022relative} the upper estimates of the projectional coderivative of the solution mapping for a parametric system were given and applied to specific linear complementarity problems and affine variational inequalities. It was shown that an equality is attainable under regularity condition when referring a smooth manifold within the domain. 

In this direction, we continue our research in analyzing the properties of projectional coderivatives relative to a smooth manifold. It turns out that, by virtue of the orthogonality between the tangent cone and the normal cone of smooth manifolds, the projected normal cone mapping is equal to an intersection and becomes outer-semicontinuous locally.  Therefore the projectional coderivative is shown to be an intersection of restricted coderivative and the tangent cone of the smooth manifold. A result in \cite{Daniilidis2011} shows that this intersection, when contains $0$ only, provides sufficiency of relative Lipschitz-like property. We revisit the sufficient and necessary conditions in \cite{Meng2020} and extend the generalized Mordukhovich criterion for this case. Such a sufficiency is proven to be a full characterization, along with other equivalences. 

The pursuit of this paper is two-fold. The first objective is to investigate the possibility of simplifying the expression of projectional coderivatives relative to some certain set: smooth manifolds and to extend the generalized Mordukhovich criterion under this setting. The other attention of this paper is paid to developing calculus rules for projectional coderivatives. We first develop the chain rule for composition mapping. Unlike the chain rule for coderivatives, here we require a stronger condition due to the projectional structure. A neater equation can be attained when both mappings are graph-convex and the relative set is a smooth manifold. We also mention two special cases: the outer or the inner mapping is single-valued. The first case can be derived naturally while the latter needs some extra effort in giving coderivative of a restricted single-valued mapping. Instead of applying the chain rule directly to obtain the sum rules, we develop the sum rules based on that of coderivatives to maintain tighter estimates considering the asymmetric nature of projectional coderivatives. Also, as the projectional coderivative involves the restricted mapping, the restriction can be imposed in different levels: either in summation mapping $S$ or the component mappings $S_i$. We develop two sum rules accordingly for user's convenience. The implication between the constraint qualifications in these two sum rules is illustrated by an example. 

The organization of the paper is as follows. \Cref{sect:Prelim} introduces the standard tools and notations in variational analysis and the tool projectional coderivatives along with the generalized Mordukhovich criterion. The geometric difference between the projectional coderivative and coderivative is also illustrated directly via an example along with a figure. Our work begins in \Cref{sect:PCandSM} by introducing some properties of smooth manifolds, mainly from the perspective of tangent cones and projections to give a fixed-point expression of projectional coderivatives relative to smooth manifolds. \Cref{sect:LipSM} extends the generalized Mordukhovich criterion and gives complete characterizations of the Lipschitz-like property relative to a smooth manifold.
\Cref{sect:chainrules} is devoted to obtaining the chain rule for projectional coderivatives and establish an equation for smooth manifolds similar to coderivatives.  Some special cases like when inner or outer layer of the function is single-valued are also discussed. Subsequently in \Cref{sect:sumrules} two sum rules are analyzed with different constraint qualifications. The difference is mainly caused by how we deal with restricting the mappings onto the set.

\section{Preliminaries}\label{sect:Prelim}

In this section, we review some notations, tools and corresponding properties extensively used throughout the paper. Most of these are standard in variational analysis and can be found in monographs \cite{mordukhovich2006variational,VaAn}. Readers who are familiar with these notations may safely skip this section. 

The closed unit ball and the unit sphere in $\mathbb{R}^n$ are denoted by $\mathbb{B}$ and $\Sph$ respectively. Given a nonempty set $C\subseteq \mathbb{R}^n$, the interior, the closure,  the boundary, and the positive hull of $C$ are denoted respectively by $\inte C$,  $\cl C$, $\bdry C$,  and $\pos C:=\{0\}\cup \{\lambda x \mid x\in C$, $\lambda>0\}$.
The projection mapping $\proj_C$ is defined by $$ {\rm proj}_C (x):=\{y\in C\mid \|y-x\|=d(x,C)\}, $$
where $d(x,C)$ is the distance from $x$ to $C$.
For a set $X\subset \mathbb{R}^n$, we denote the projection of $X$ onto $C$  by
$$
{\rm proj}_C X:=\{y\in C\mid \exists x\in X,\text{ s.t. } d(x,y) =d(x,C)\}.
$$
If $C=\emptyset$, by convention we set that $d(x, C):=+\infty$, ${\rm proj}_C(x):=\emptyset$, and ${\rm proj}_C X:=\emptyset$.

Let $x\in C$. We use $T_C(x)$ to denote the  tangent/contingent cone to $C$ at $x$, given by
$$T_C(x) = \limsup_{t\ \searrow  \ 0} \frac{C - x}{t}.$$
The regular/Fr\'{e}chet  normal cone, $\widehat{N}_C(x)$, is the polar cone of $T_C(x)$, defined by 
$$\widehat{N}_C(x) = \left\{ v\in \mathbb{R}^n \left\vert \ \limsup_{x'\xrightarrow[\neq]{C}x}  \frac{\langle v, x'-x \rangle}{\|x'- x\|} \leq 0 \right. \right\}.  $$
Here $x'\xrightarrow[\neq]{C}x  $ means $x' \to x$, $x'\in C$, $x'\neq x$.
The (basic/limiting/Mordukhovich) normal cone to $C$ at $x$, $N_C(x)$, is defined via the outer limit of $\widehat{N}_C$ 
as 
$$N_C(x) := \left\{ v \in \mathbb{R}^n \big\vert \exists \text{ sequences }x_k \xrightarrow[]{C} x,  \ v_k \rightarrow v, \ v_k \in \widehat{N}_C(x_k) ,\  \forall k\right\}. $$
We say that $C$ is locally closed at a point $x\in C$ if $C\cap U$ is closed for some closed neighborhood $U \in \mathcal{N}(x)$. 
$C$ is said to be regular at $x$ in the sense of Clarke if it is locally closed at $x$ and $\widehat{N}_C(x)=N_C(x)$.
For any $x \notin C$, we set by convention $T_C (x),  \widehat{N}_C(x), N_C(x)$ are all empty sets. 

Let $f:\mathbb{R}^n\to \overline{\mathbb{R}}:=\mathbb{R}\cup\{\pm \infty\}$ be an extended real-valued function and let $\bar{x}$ be  a point with $f(\bar{x})$ finite. The vector $v\in \mathbb{R}^n$ is a regular/Fr\'{e}chet subgradient of $f$ at $\bar{x}$, written $v\in \widehat{\partial} f(\bar{x})$, if
\[
f(x)\geq f(\bar{x})+\langle v, x-\bar{x}\rangle+o(||x-\bar{x}||).
\]
The vector $v\in \mathbb{R}^n$ is a (general/basic) subgradient of $f$ at $\bar{x}$, written $v\in \partial f(\bar{x})$, if there exist sequences $x_k\to \bar{x}$ and $v_k\to v$ with $f(x_k)\to f(\bar{x})$ and $v_k\in \widehat{\partial} f(x_k)$.  The subdifferential set $\partial f(\bar{x})$ is also referred to as limiting/Mordukhovich subdifferential.

\subsection{Notations for set-valued mappings}

For a set-valued mapping $S:\mathbb{R}^n\rightrightarrows \mathbb{R}^m$, we denote by
$\gph S:=\{(x, u)\mid u\in S(x)\}$ the graph of $S$ and $\dom S:=\{x\mid S(x)\not=\emptyset\}$ the domain of $S$. $S$ is said to be positively homogeneous if $\gph S$ is a cone. If $S$ is a positively homogeneous mapping, the outer norm  of $S$ is denoted and defined by
  \begin{equation}\label{eq:def-outernorm}
       \lvert S \rvert^+:=\sup_{x\in \mathbb{B}}\sup_{u\in S(x)}\|u\|.
  \end{equation}
The  (normal) coderivative and the regular/Fr\'{e}chet coderivative 
  of $S$ at $\bar{x}$ for any $\bar{u}\in S(\bar{x})$ are respectively the mapping $D^*S(\bar{x}\mid \bar{u}):\mathbb{R}^m\rightrightarrows \mathbb{R}^n$ defined by
  \begin{equation}\label{eq:def:Coderivatives}
      x^*\in D^*S(\bar{x}\mid \bar{u})(u^*)\Longleftrightarrow (x^*, -u^*)\in N_{\gph S}(\bar{x}, \bar{u}),
  \end{equation}
  the mapping $\widehat{D}^*S(\bar{x}\mid \bar{u}):\mathbb{R}^m\rightrightarrows \mathbb{R}^n$ defined by
  \[
  x^*\in \widehat{D}^*S(\bar{x}\mid \bar{u})(u^*)\Longleftrightarrow (x^*, -u^*)\in \widehat{N}_{\gph S}(\bar{x}, \bar{u}).
  \]
  Both of these mappings are positively homogeneous. 
 
   For a set $X\subset \mathbb{R}^n$, we denote the restricted mapping of $S$ on $X$ by
  \[
  S\vert_X:=\left\{
  \begin{array}{ll}
    S(x) & \mbox{if}\;x\in X, \\[0.2cm]
    \emptyset & \mbox{if}\;x\not\in X.
  \end{array}
    \right.
  \]
  It is clear to see that
  \[
  \gph S\vert_X=\gph S\cap (X\times \mathbb{R}^m)\quad\mbox{and}\quad \dom  S\vert_X=X\cap \dom S.
  \]
As mentioned, one of the primary goals of the paper is to investigate relative Lipschitz-like property. Accordingly, we introduce some properties relative to a set. 
\begin{definition}[local boundedness relative to a set,{\cite[p. 162]{VaAn}}]
For a mapping $S: \mathbb{R}^n \rightrightarrows \mathbb{R}^m$, a closed set $X \subset \mathbb{R}^n$ and a given point $\bar{x} \in X$, if for some neighborhood $V\in \mathcal{N}(\bar{x})$, $S(V\cap X)$ is bounded, we say $S$ is locally bounded relative to $X$ at $\bar{x}$. Such definition is equivalent to local boundedness of $S\vert_X$ at $\bar{x}$, where $S\vert_X$ means the mapping $S$ restricted to $X$.
\end{definition}

Similar to definition of regularity in \cite[Defnition 6.4]{VaAn}, here we introduce a local version:
\begin{definition}
  For a set $C\subseteq \mathbb{R}^n$, we say $C$ is regular around $\bar{x} \in C$ (in the sense of Clarke) if it is locally closed around $\bar{x}$ and there exists a neighborhood $X$ of $\bar{x}$, such that for any $x\in X \cap C$, $\widehat{N}_C (x) = N_C(x)$.
\end{definition}
\begin{definition}[Outer semicontinuity relative to a set, {\cite[Definition 5.4]{VaAn}}] A set-valued mapping $S: \mathbb{R}^n \rightrightarrows \mathbb{R}^m$ is outer semicontinuous (osc) at $\bar{x}$ relative to $X$ if $\bar{x} \in X$ and
   $\limsup_{\tiny x\xrightarrow[]{X}\bar{x}}S(x) = S(\bar{x}).$
Such a definition is equivalent to outer semicontinuity of $S\vert_X$ as $$\displaystyle \limsup_{\tiny x\xrightarrow[]{X}\bar{x}} S(x) = \limsup_{\tiny x\rightarrow  \bar{x}} S\vert_X(x) = S\vert_X(\bar{x}).$$
\end{definition}
\subsection{Relative Lipschitz-like property and generalized Mordukhovich criterion}
Next we present a central role in stability of $S$, the Lipschitz-like property, and the corresponding tool for the criterion. 
\begin{definition}[Lipschitz-like property relative to a set, {\cite[Definition 9.36]{VaAn}}]\label{def-lip-like}
A mapping $S: \mathbb{R}^n\rightrightarrows  \mathbb{R}^m$ has the Lipschitz-like property relative to $X$ at $\bar{x}$ for $\bar{u}$, where $\bar{x}\in X$ and $\bar{u}\in S(\bar{x})$, if $\gph S$ is locally closed at $(\bar{x}, \bar{u})$ and there are neighborhoods $V\in \mathcal{N}(\bar{x})$, $W\in \mathcal{N}(\bar{u})$, and a constant $\kappa \in \mathbb{R}_+$ such that
\begin{equation}\label{Ap-def}
S(x')\cap W\subset S(x)+\kappa\|x'-x\|\mathbb{B}, \quad \forall x,x'\in X\cap V.
\end{equation}
The graphical modulus of $S$ relative to $X$ at $\bar{x}$ for $\bar{u}$ is then
\[
\begin{array}{ll}
\lip_X S(\bar{x}\mid\bar{u}):=\inf\;\{\;\kappa\geq 0\mid  \exists V & \in \mathcal{N}(\bar{x}), W\in \mathcal{N}(\bar{u}),\;\mbox{such that}\\[0.2cm]
&S(x')\cap W\subset S(x)+\kappa\|x'-x\|\mathbb{B}, \quad \forall x,x'\in X\cap V\;\}.
\end{array}
\]
\end{definition}
The property with $V$ in place of $X\cap V$ in \eqref{Ap-def} is the Lipschitz-like property along with the graphical modulus $\lip S(\bar{x} \mid \bar{u})$. 
To characterize the relative Lipschitz-like property, in \cite{Meng2020}, a new tool, the projectional coderivative, is introduced. 

\begin{definition}[{\cite[Definition 2.2]{Meng2020}}]\label{def:projcode}
  $D^*_X S(\bar{x} \mid \bar{u}) : \mathbb{R}^m \rightrightarrows \mathbb{R}^n$  of $S: \mathbb{R}^n \rightrightarrows \mathbb{R}^m$ at $\bar{x} \in X$ for any $\bar{u} \in S(\bar{x})$ with respect to $X$ is defined as
  \begin{equation}\label{DefofProjCode}
  t^*\in D^*_{X}S(\bar{x}\mid \bar{u})(u^*)\Longleftrightarrow (t^*, -u^*)\in \limsup_{\tiny (x,u)\xrightarrow[]{\gph S\vert_X}(\bar{x}, \bar{u})} {\rm proj}_{T_X(x)\times \mathbb{R}^m}N_{\gph S\vert_X}(x,u).
  \end{equation}
\end{definition}
Here we give an example of the calculation on projectional coderivatives.

\begin{figure}[ht]
\centering
\includegraphics[scale=0.32]{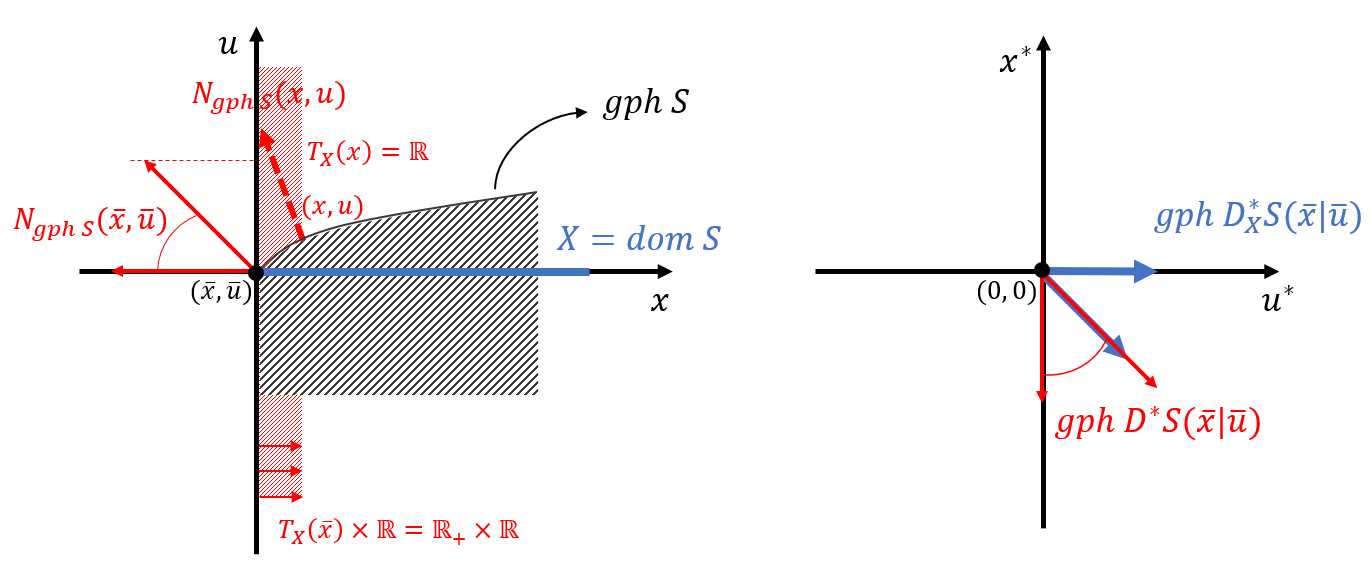}
\caption{The projectional coderivatives in \cref{ex:codevsprojcode}}
\end{figure}
\begin{example}\label{ex:codevsprojcode}
Let $$S(x) = \begin{cases}
    (-\infty, \frac{2}{1 + e^{-x}} -1 ], & \text{ if } x \geq 0. \\
    \emptyset, & \text{ else. }
\end{cases} $$
We consider the reference pair $(\bar{x}, \bar{u}) = (0, 0) \in \gph S$. $X = \dom S = \mathbb{R}_+$. Thus $\gph S\vert_X = \gph S$. For $(x,u) \xrightarrow[]{\gph S} (\bar{x},\bar{u}) $ we have the following three cases (trivial case omitted):
\begin{description}
    \item[1.] $x \in \mathbb{R}_{++}$: $T_X(x) = \mathbb{R}$ and 
    $$\displaystyle {\rm proj}_{T_X(x)\times \mathbb{R}} N_{\gph S}(x,u)  = N_{\gph S}(x,u) \rightarrow \mathbb{R}_+(-1, 1).$$
    \item[2.] $x =0$ with $u \in \mathbb{R}_{--}:$ $T_X (x) = \mathbb{R}_+ $ and $$\displaystyle {\rm proj}_{T_X(x)\times \mathbb{R}} N_{\gph S}(x,u)  = {\rm proj}_{\mathbb{R}_+ \times \mathbb{R}} \left( \mathbb{R}_- \times \{0\} \right) =  \{(0,0)\}.$$
    \item[3.] $(x, u ) = (\bar{x},\bar{u})$: $T_X (\bar{x}) = \mathbb{R}_+$, $N_{\gph S}(\bar{x}, \bar{u}) = \left\{ (x^*, u^*) \mid 0 \leq u^* \leq -x^* \right\}. $
    $$\displaystyle {\rm proj}_{T_X(\bar{x})\times \mathbb{R}} N_{\gph S}(\bar{x},\bar{u}) = \{0\} \times \mathbb{R}_+.$$
\end{description}

Thus by definition \eqref{DefofProjCode} we have 
\begin{align*}
   \gph D^*S(\bar{x}\mid\bar{u}) & = \left\{(u^*,x^*)\mid 0\leq u^* \leq -x^* \right\} \\
   \gph D^*_X S(\bar{x}\mid\bar{u}) & = \mathbb{R}_+\times \{0\} \cup \mathbb{R}_+(1, -1).
\end{align*}
\end{example}
Given this new tool, a handy test for the Lipschitz-like property relative to a closed and convex set is developed similarly to the Mordukhovich criterion, and is named as the generalized Mordukhovich criterion. 
\begin{theorem}[generalized Mordukhovich criterion, {\cite[Theorem 2.4]{Meng2020}}]\label{thm-genMordcri}
 Consider $S: \mathbb{R}^n\rightrightarrows  \mathbb{R}^m$, $\bar{x}\in X\subset \mathbb{R}^n$ and $\bar{u}\in S(\bar{x})$. Suppose that $\gph S$ is locally closed at $(\bar{x}, \bar{u})$ and that $X$ is closed and convex.  Then $S$ has the Lipschitz-like property relative to $X$ at $\bar{x}$ for $\bar{u}$ if and only if $D^*_{X}S(\bar{x}\mid\bar{u})(0)=\{0\}$
 or equivalently $\lvert D^*_X S(\bar{x}\mid\bar{u})\rvert ^+<+\infty$. In this case, 
${\rm lip}_X S(\bar{x}\mid\bar{u})=\lvert D^*_X S(\bar{x}\mid\bar{u}) \rvert^+.$
\end{theorem}
From the definition \ref{def:projcode}, we can see that when $\bar{x}\in \inte X$, the projectional coderivative reduces to coderivatives naturally. Accordingly, the generalized Mordukhovich criterion becomes the Mordukhovich criterion as well (see \cite{mordukhovich1992sensitivity, VaAn}).

\section{Projectional coderivatives and properties of smooth manifolds}\label{sect:PCandSM}

The definition of projectional coderivatives involves normal cones, the projection, together with the process of taking outer limit (see \eqref{DefofProjCode}). Thus the calculation of it can be complicated and deviates from the commonly adopted coderivatives.
In this section, we are concerned with a multifunction $S: \mathbb{R}^n \rightrightarrows \mathbb{R}^m$ and a closed set $X \subseteq \mathbb{R}^n$ and the projectional coderivative of $S$ relative to a smooth manifold. We first introduce some properties of projection and some natural observations of projectional coderivatives. The special setting of $X$: smooth manifolds, allows us to simplify the expression $D^*_X S$ to a fixed-point expression. Some examples are given to illustrate the calculations. Below we first present an observation on the connection between projectional coderivatives and coderivatives, which is introduced in \cite{yao2022relative}. 
\begin{lemma}\label{lm:ProjCodeInverseAt0-Incl}\cite[Corollary 2.5]{yao2022relative}
  For a set valued-mapping $S: \mathbb{R}^n \rightrightarrows \mathbb{R}^m$, and a closed set $X \subseteq \mathbb{R}^n$, for any $(\bar{x}, \bar{u})\in \gph S\vert_X$,
  \begin{equation}\label{ProjCodeInverseAt0-Incl}
    D^*S\vert_X (\bar{x} \mid \bar{u})^{-1} (0) \subseteq D^*_X S (\bar{x} \mid \bar{u})^{-1}(0).
  \end{equation}
\end{lemma}
By definition of projectional coderivatives (Definition \ref{def:projcode}), it involves projecting normal cone of $\gph S$ onto $T_X(x) \times \mathbb{R}^n$ for all neighboring points. The projection onto the tangent cone enjoys continuity when the tangent cone mapping is also continuous, in a manner that the expression of projectional coderivative can be refined as a fixed point one.
Before introducing the exact form, we first present some basic properties of a smooth manifold in the following proposition. 
In what follows, let $X$ be a $d$-dimensional smooth manifold in $\mathbb{R}^n$ around the point $\bar{x}\in X$, in the sense that $X$ can be represented relative to an open neighborhood $O\in \mathcal{N}(\bar{x})$ as the set of solutions to $F(x)=0$, where $F:O\rightarrow \mathbb{R}^{n-d}$ is a smooth (i.e., $\mathcal{C}^1$) mapping with $\nabla F(\bar{x})$ of full rank $n-d$. This definition is borrowed from  \cite[Example 6.8]{VaAn}. For more thorough details of smooth manifold, see the monograph by \cite{lee2013smooth}.

\begin{proposition}[Basic properties of smooth manifolds]\label{Prop:SM-basic}
Let $X \in \mathbb{R}^n$ be a smooth manifold around $\bar{x}$. We have the following basic properties. 
\begin{description}
  \item[(a)] $X$ is regular around $\bar{x}$. The tangent and normal cones to $X$ at any $x$ being close enough to $\bar{x}$ are linear subspaces orthogonally complementary to each other, namely
  \[
  T_X(x)=\{w\in \mathbb{R}^n\mid \nabla F(x) w=0\}\quad\mbox{and}\quad N_X(x)=\{\nabla F(x)^* y\mid y\in \mathbb{R}^{n-d}\}.
  \]
  Moreover,
  \begin{equation*}
  {\rm proj}_{T_X(x)}(x^*)=\left[I-\nabla F(x)^*\left(\nabla F(x)\nabla F(x)^*\right)^{-1}\nabla F(x)\right]x^*, \quad\forall x^*.
  \end{equation*}
  \item[(b)] For any $x$ being close enough to $\bar{x}$ in $X$ and $x^*\in \mathbb{R}^n$, it holds that
  \[
  T_X(x_k)\to T_X(x)\quad\mbox{and}\quad {\rm proj}_{T_X(x_k)}(x^*_k)\to {\rm proj}_{T_X(x)}(x^*),
  \]
  where $\{x_k\}$ and $\{x_k^*\}$ are  two sequences such that $x_k {\xrightarrow[]{X}}x$ and $x_k^*\to x^*$.
  \item [(c)] For any $\varepsilon>0$, there exists some $\delta>0$ such that
  \[
  \displaystyle\left\|\;\frac{y-x}{\|y-x\|}-{\rm proj}_{T_X(x)}\left( \frac{y-x}{\|y-x\|} \right)\;\right\| \leq \varepsilon
  \]
  holds for all $x,y\in X\cap \mathbb{B}_{\delta}(\bar{x})$ with $x\not=y$.
\end{description}
\end{proposition}
\begin{proof}
  The properties in (a) then follows readily from \cite[Exercise 6.7, Example 6.8]{VaAn} and \cite[p. 365]{lang2016introduction}.

  For (b), with regularity, by \cite[Corollary 6.29 (b)]{VaAn}, 
  $ T_X(x) = \widehat{T}_X(x) = \liminf_{\tiny x'\xrightarrow[]{X}x}T_X(x') .$
  To prove continuity relative to $X$, it remains to prove 
  $  \limsup_{\tiny x'\xrightarrow[]{X}x}T_X(x') \subseteq T_X (x)$. 
  For $  w \in \limsup_{\tiny x'\xrightarrow[]{X}x}T_X(x')$, it is equivalent that there exist sequences $$x_k \xrightarrow[]{X} x \text{ and } w_k \in T_X(x_k) = \left\{ w\in \mathbb{R}^n \mid \nabla F(x_k) w =0 \right\}$$ such that $w_k \rightarrow w$.
  Given $\nabla F(x_k) \rightarrow \nabla F(x)$ when $x_k \rightarrow x$ and $\nabla F(x_k) w_k = 0$ , then  $\nabla F(x_k) w_k  \rightarrow \nabla F(x) w =0$ when $k\rightarrow \infty $, which shows $w \in T_X(x)$ and thus $T_X(\cdot)$ is continuous at $x$ relative to $X$  and always convex-valued. By \cite[Example 5.35]{VaAn}, we have 
  $${\rm proj}_{T_X(x_k)}(x^*_k)\to {\rm proj}_{T_X(x)}(x^*) \text{ for } x_k \xrightarrow[]{X}x \text{ and } x^*_k \rightarrow x^*.$$
  
  It remains to show (c). Suppose by contradiction that there exist  some $\varepsilon_0>0$  and some sequences $x_k,y_k { \xrightarrow[]{X}}\bar{x}$ with $x_k\not= y_k$  such that
  \begin{equation}\label{ProjPointClose-Contra}
    \displaystyle\left\|\;\frac{y_k-x_k}{\|y_k-x_k\|}-{\rm proj}_{T_X(x_k)} \left(\frac{y_k-x_k}{\|y_k-x_k\|} \right) \;\right\|> \varepsilon_0, \quad\forall k.
  \end{equation}
  Without loss of generality, we can assume that $x_k, y_k\in O$  with $\nabla F(x_k)$ of full rank $n-d$ for all $k$,  and that there is some $w\in \Sph$ such that $\frac{y_k-x_k}{\|y_k-x_k\|}\rightarrow w.$
It then follows that
\begin{equation}\label{SM-MeanValueApp}
 \int_0^1 \nabla F(\tau y_k+(1-\tau) x_k)\;d\tau \cdot \frac{y_k-x_k}{\|y_k-x_k\|}=\frac{F(y_k)-F(x_k)}{\|y_k-x_k\|}=0,\quad \forall k,
\end{equation}
where the integral of a matrix is to be understood componentwise. Applying componentwise the first mean value theorem for definite integrals, we have
\[
\int_0^1 \nabla F(\tau y_k+(1-\tau) x_k)\;d\tau \rightarrow \nabla F(\bar{x}),
\]
and hence
\[
\int_0^1 \nabla F(\tau y_k+(1-\tau) x_k)\;d\tau \cdot \frac{y_k-x_k}{\|y_k-x_k\|}\rightarrow \nabla F(\bar{x})w.
\]
In view of (\ref{SM-MeanValueApp}), we have $ \nabla F(\bar{x})w=0. $
As  $\nabla F(x_k)$ is of full row rank, we have
$T_X(x_k)=\{w\mid \nabla F(x_k) w=0\},$ 
and hence
\begin{align*}
      & \frac{y_k-x_k}{\|y_k-x_k\|}-{\rm proj}_{T_X(x_k)} \left(\frac{y_k-x_k}{\|y_k-x_k\|} \right) \\
    = &    \  \nabla F(x_k)^*\left(\nabla F(x_k)\nabla F(x_k)^*\right)^{-1}\nabla F(x_k) \cdot \frac{y_k-x_k}{\|y_k-x_k\|}  \\
     \rightarrow  & \  \nabla F(\bar{x})^*\left(\nabla F(\bar{x})\nabla F(\bar{x})^*\right)^{-1}\nabla F(\bar{x})w=0,
\end{align*}
contradicting to (\ref{ProjPointClose-Contra}). This completes the proof. 
\end{proof}
These local properties of smooth manifolds facilitate us in reducing the projectional coderivative expression of $S$ relative to a smooth manifold to a fixed-point expression. In the next proposition we show how the normal cone of $\gph S$ restricted to $X$ interacts with the tangent cone $T_X$. 
\begin{proposition}[Projectional coderivatives of a set-valued mapping on a smooth manifold]\label{Prop:SM-ProjCode}
Consider $S: \mathbb{R}^n\rightrightarrows  \mathbb{R}^m$ and  $\bar{u}\in S(\bar{x})$.   Suppose that $\gph S$ is locally closed at $(\bar{x}, \bar{u})$ and $X$ is a smooth manifold around $\bar{x}$ with $\bar{x}\in X$. 
The following properties hold
 for all $(x,u)$ close enough to $(\bar{x},\bar{u})$ in $\gph S\vert_X$:
 \begin{description}
   \item[(a)] ${\rm proj}_{T_X(x)\times \mathbb{R}^m} \widehat{N}_{\gph S\vert_X}(x,u)=\widehat{N}_{\gph S\vert_X}(x,u)\cap \left( T_X(x)\times \mathbb{R}^m\right) $.
   \item[(b)] ${\rm proj}_{T_X(x)\times \mathbb{R}^m} N_{\gph S\vert_X}(x,u)=N_{\gph S\vert_X}(x,u)\cap \left( T_X(x) \times \mathbb{R}^m \right)$.
   \item[(c)] $D^*_X S(x\mid u)(u^*)={\rm proj}_{T_X(x)}D^*S\vert_X(x\mid u)(u^*)=D^*S\vert_X(x\mid u)(u^*)\cap T_X(x),\quad \forall u^*$.
   \end{description}
\end{proposition}
\begin{proof} In what follows, let $(x,u)$ be close enough to $(\bar{x},\bar{u})$ in $\gph S\vert_X$ such that the properties in  \Cref{Prop:SM-basic} (a) and (b) holds.

To prove (a), it suffices to show
 \begin{equation}\label{incl:proj-RNC}
    {\rm proj}_{T_X(x) \times \mathbb{R}^m} \widehat{N}_{\gph S\vert_X} (x,u) \subset \widehat{N}_{\gph S\vert_X} (x,u)\cap \left( T_X(x)\times \mathbb{R}^m \right).
  \end{equation}
Let $(y^*, u^*)$ belong to the left-hand side of (\ref{incl:proj-RNC}).  Then there exists some $x^*$ such that $y^* = {\rm proj}_{T_X(x)}(x^*)$ and  $(x^*, u^*) \in   \widehat{N}_{\gph S\vert_X} (x,u)$. Then by definition we have
  \begin{equation}\label{deffromRNC}
  \limsup_{\tiny (x', u')\xrightarrow[(x',u')\not=(x,u)]{\gph S\vert_X}(x, u)} \frac{\langle (x^*, u^* ),\;(x'-x, u'-u)\rangle}{\|(x'-x, u'-u)\|}\leq 0.
  \end{equation}
  Let $z^*:={\rm proj}_{N_X(x)}( x^*)$. As $N_X(x)$ is a linear subspace (see \Cref{Prop:SM-ProjCode} (a)), we  have $\pm z^* \in N_X(x)$. This implies that $
  \lim_{\tiny x'\xrightarrow[x'\neq x]{X}x} \frac{\langle z^*,x'-x\rangle}{\|x'-x\|}=0, $
  and hence that
  \begin{equation}\label{deffromRNC2}
  \lim_{\tiny (x', u')\xrightarrow[(x',u')\neq (x, u)]{\gph S\vert_X}(x, u)} \frac{\langle z^*,x'-x\rangle}{\|(x'-x, u'-u)\|}= 0.
  \end{equation}
By \cite[Example 12.22]{VaAn}, $x^*=y^*+z^*$ and such a representation is unique. It follows from (\ref{deffromRNC}) and (\ref{deffromRNC2}) that
\[
\limsup_{\tiny (x', u')\xrightarrow[(x', u')\neq (x, u)]{\gph S\vert_X}(x, u)} \frac{\langle (y^*, u^* ),\;(x'-x, u'-u)\rangle}{\|(x'-x, u'-u)\|}
 \leq 0,
 \]
 which amounts to that $(y^*, u^*)\in \widehat{N}_{\gph S\vert_X} (x,u)$. From the fact that $y^*\in T_X(x)$,  it then follows that  $(y^*, u^*)$ belongs to the right-hand side of (\ref{incl:proj-RNC}).

 To prove (b), it suffices to show
  \begin{equation}\label{incl:proj-NC}
      {\rm proj}_{T_X(x) \times \mathbb{R}^m} N_{\gph S\vert_X} (x,u)\subset N_{\gph S\vert_X} (x,u)\cap \left( T_X(x)\times \mathbb{R}^m \right).
  \end{equation}
  Let $(y^*, u^*)$ belong to the left-hand side of (\ref{incl:proj-NC}). Then there exists $x^*$ such that $y^*={\rm proj}_{T_X(x)}(x^* )$ and  $(x^*, u^*) \in N_{\gph S\vert_X}(x,u)$.
  By definition there are sequences $(x_k, u_k) \xrightarrow[]{\gph S\vert_X} (x, u)$ and $(x^*_k , u^*_k)\in \widehat{N}_{\gph S\vert_X}(x_k, u_k)$ such that $(x^*_k , u^*_k) \rightarrow (x^*, u^*)$. In view of (a), we have for all sufficiently large $k$,
  \[
   \left({\rm proj}_{T_X(x_k)}(x^*_k),\, u^*_k\right)\in \widehat{N}_{\gph S\vert_X}(x_k, u_k).
  \]
  By \Cref{Prop:SM-basic} (b), we have ${\rm proj}_{T_X(x_k)}(x^*_k)\to {\rm proj}_{T_X(x)}(x^* )=y^*$. Then by definition we have $(y^*, u^*)\in N_{\gph S\vert_X}(x, u)$. 
 From the fact that $y^*\in T_X(x)$,  it then follows that $(y^*, u^*)$ belongs to the right-hand side of (\ref{incl:proj-NC}).
  Let $u^*\in \mathbb{R}^m$. From (b) and the definitions of coderivatives \eqref{eq:def:Coderivatives} and projectional coderivatives \eqref{DefofProjCode}, we get
  \[
    D^*_XS(x\mid u)(u^*)\supset {\rm proj}_{T_X(x)}D^*S\vert_X(x\mid u)(u^*)=D^*S\vert_X(x\mid u)(u^*)\cap T_X(x).
  \]
  To show (c), it suffices to show
   \begin{equation}\label{incl:projCode}
       D^*_XS(x\mid u)(u^*)\subset D^*S\vert_X(x\mid u)(u^*)\cap T_X(x).
   \end{equation}
  Let $y^*$ belong to the left-hand side of (\ref{incl:projCode}). Then by definition there are some $(x_k,u_k)\xrightarrow[]{\gph S\vert_X}(x,u)$ and
  $x_k^*\in D^*S\vert_X(x_k\mid u_k)(u_k^*)$ such that $u_k^*\to u^*$ and $y_k^*:={\rm proj}_{T_X(x_k)}(x_k^*)\to y^*$. By (b), we have for all sufficiently large $k$,
  \[
   y_k^*\in D^*S\vert_X(x_k\mid u_k)(u_k^*) \cap T_X(x_k),
  \]
implying that $y^*\in D^*S\vert_X(x\mid u)(u^*)$.  As $y_k^*\in T_X(x_k)$, we get from \Cref{Prop:SM-basic} (b) that $y^*\in T_X(x)$. That is, $y^*$ belongs to the right-hand side of (\ref{incl:projCode}). This completes the proof.
\end{proof}

Next we give a simple example for geometric interpretation of the results in \Cref{Prop:SM-ProjCode}. 
\begin{example}
Consider a multifunction $S: \mathbb{R}^2 \rightrightarrows \mathbb{R}^2$ defined as 
$$S\left( (x_1,x_2)^\top \right)  = \begin{cases}
   \{(0, -x_2)^\top \}, &  x_1\neq 0 \\
   \mathbb{R}^2, & x_1 = 0
\end{cases}. $$ For $X =\mathbb{R}\times \{1\} \subseteq \dom S = \mathbb{R}^2$ and $\bar{x} =(0,1)^\top$, $\bar{u} =(0,0)^\top$, $(\bar{x},\bar{u}) \in \gph S\vert_X$. By calculation we have 
$N_{\gph S\vert_X}(\bar{x}\mid \bar{u}) = \mathbb{R}^2 \times \{0_2\}$, $T_X(\bar{x}) = \mathbb{R}\times \{0\}$
and 
$$D^*S\vert_X(\bar{x}\mid \bar{u})(u^*)  = \begin{cases}
  \mathbb{R}^2,    & u^* =(0, 0)^\top\\
  \emptyset, & u^* \neq (0, 0)^\top
\end{cases}.$$
Then we can see that 
\begin{equation*}
\begin{aligned}
        & {\rm proj}_{T_X(x)\times \mathbb{R}^2} N_{\gph S\vert_X}(\bar{x}\mid \bar{u})= {\rm proj}_{\mathbb{R}\times \{0\} \times \mathbb{R}^2}\left( \mathbb{R}^2 \times \{0_2\} \right) \\
    = & \ \mathbb{R}\times \{0_3\} =  N_{\gph S\vert_X}(\bar{x}\mid \bar{u})\cap \left( T_X(\bar{x}) \times \mathbb{R}^m \right)
\end{aligned}
\end{equation*}
and 
\begin{equation*}
    \begin{aligned}
       D^*_X S(\bar{x}\mid \bar{u})(u^*)={\rm proj}_{T_X(\bar{x})}D^*S\vert_X(\bar{x}\mid \bar{u})(u^*) & = 
       \begin{cases}
          \mathbb{R} \times \{0\},    & u^* =(0, 0)^\top\\
          \emptyset, & u^* \neq (0, 0)^\top
       \end{cases}  \\ 
      & =  D^*S\vert_X(\bar{x}\mid \bar{u})\cap T_X(\bar{x}). 
    \end{aligned} 
\end{equation*}
\end{example}
For coderivatives defined as in \eqref{eq:def:Coderivatives}, we know that by outer semicontinuity of normal cone mappings$$ D^*S(\bar{x}\mid\bar{u}) = \glimsup_{(x,u)\xrightarrow[]{\gph S}(\bar{x},\bar{u})  }D^* S(x\mid u).$$  It is natural to ask if projectional coderivatives have such a property as well. The coming corollary is a natural observation from Propositions \ref{Prop:SM-basic} and \ref{Prop:SM-ProjCode}, and \cite[Lemma 3.3]{yao2022relative}. 
\begin{corollary}\label{coro:Osc-ProjCode}
Consider $S: \mathbb{R}^n\rightrightarrows  \mathbb{R}^m$ and  $\bar{u}\in S(\bar{x})$.   Suppose that $\gph S$ is locally closed at
 $(\bar{x}, \bar{u})$, $\bar{x}\in X$ and $X$ is a smooth manifold around $\bar{x}$ with $X\subseteq \dom S$. Then the mapping 
 $(x,u) \mapsto \proj_{T_X(x) \times \mathbb{R}^m} N_{\gph S\vert_X} (x,u)$ is outer semicontinuous relative to $\gph S\vert_X$ at $(\bar{x}, \bar{u})$ and 
 \begin{equation}\label{eqn:osc-projcode}
     D^*_X S(\bar{x} \mid \bar{u}) = \glimsup_{(x,u)\xrightarrow[]{\gph S\vert_X}(\bar{x},\bar{u})  }D^*_X S(x\mid u) = {\rm proj}_{T_X(\bar{x})} D^* S\vert_X(\bar{x} \mid \bar{u}) .
 \end{equation}
\end{corollary}
\begin{proof}
 This result is simply an application of \Cref{Prop:SM-basic}, \Cref{Prop:SM-ProjCode} and the outer semicontinuity of normal cone mapping. 
\end{proof}

We next show by a smooth function that the calculation of projectional coderivative may not be as simple as the coderivatives as the projection of normal cone does not enjoy outer semicontinuity unless the set $X$ has some nice structures. For example, when $X$ is a closed half-space, the projectional coderivative becomes either a line segment or a set containing two points (see \cite[Remark 2.2]{Meng2020}). 
\begin{lemma}
[Projectional coderivatives of a smooth function] \label{lm:ProjCodeofSmoothFunc}
  For $F: \mathbb{R}^n \rightarrow \mathbb{R}^m$ being smooth and single-valued on $\mathbb{R}^n$ and a closed set $X\subseteq \mathbb{R}^n$, for any $\bar{x} \in \bdry X$,
  \begin{equation}\label{ProjCodeofSmoothFunc}
    D^*_X F(\bar{x})(y) =  \limsup_{\tiny x\xrightarrow[]{X}\bar{x}, y'\xrightarrow[]{} y} \left\{   {\rm proj}_{T_X(x)}\left( \nabla F(x)^*y' + w\right) \mid w \in N_X(x) \right\}.
  \end{equation}
   If furthermore $X$ is regular around $\bar{x}$,
  \begin{equation}\label{ProjCodeofSmoothFunc2}
    D^*_X F(\bar{x})(y) =  \limsup_{\tiny x\xrightarrow[]{X}\bar{x}} \left\{   {\rm proj}_{T_X(x)}\left( \nabla F(x)^*y + w\right) \mid w \in N_X(x) \right\}.
  \end{equation}
  In particular when $X$ is a smooth manifold around $\bar{x}$,
  \begin{equation}\label{ProjCodeofSmoothFunc-fixed}
    D^*_X F(\bar{x})(y) = {\rm proj}_{T_X(\bar{x})}\left( \nabla F(\bar{x})^*y \right).
  \end{equation}
\end{lemma}
\begin{proof}
For smooth mapping $F$ defined on $\mathbb{R}^n$, we can always find an open set $O \supseteq X$ such that $F$ remains smooth on $O$. In this way, $\nabla F\vert_X (x) = \nabla F(x)$ for any $x \in X$.
Then for any $x\in X$ and \cite[Example 8.34]{VaAn},
$$N_{\gph F}(x, F(x) ) =  \left\{ ( \nabla F(x)^* y, -y ) \mid  y\in \mathbb{R}^m \right\}.$$
By expressing $F\vert_X = F + \delta_X$ where $\delta_X$ is the indicator function of $X$, and by \cite[Exercise 10.43]{VaAn}, we have
\begin{equation}\label{eqn:singlesmoothonX}
    D^*F\vert_X (x) (y) = N_X(x) + \nabla F(x)^*y , \ \forall y \in \mathbb{R}^m. 
\end{equation}
That is, for all $x\in X$,
$$ N_{\gph F\vert_X} (x, F(x))  = \left\{ ( \nabla F(x)^* y+ w, -y ) \mid y\in \mathbb{R}^m, w\in N_X(x) \right\} .$$
From definition of projectional coderivative (\ref{DefofProjCode}),
\begin{equation*}
\begin{aligned}
    &  \  t \in D^*_{X}F(\bar{x})(y) \\
    \Longleftrightarrow & \ (t, -y)\in \limsup_{\tiny x\xrightarrow[]{X}\bar{x}}  {\rm proj}_{T_X(x) \times \mathbb{R}^m }N_{\gph F\vert_X}(x, F(x)) \\
    \Longleftrightarrow & \ (t, -y)\in \limsup_{\tiny x\xrightarrow[]{X}\bar{x}}  {\rm proj}_{T_X(x) \times \mathbb{R}^m }\left\{ ( \nabla F(x)^* y' + w, -y' ) \mid  y'\in \mathbb{R}^m, \ w\in N_X(x) \right\}.
\end{aligned}
\end{equation*}

Therefore we have (\ref{ProjCodeofSmoothFunc}).

For $ t\in  \displaystyle \limsup_{\tiny x\xrightarrow[]{X}\bar{x}, y'\xrightarrow[]{} y} \left\{   {\rm proj}_{T_X(x)}\left( \nabla F(x)^*y' + w\right), \ w \in N_X(x) \right\}$,
there exist sequences $x_k \xrightarrow[]{X} \bar{x}$, $w_k \in N_X(x_k)$, $y_k \in \mathbb{R}^m$  and $t_k \in  {\rm proj}_{T_X(x_k)}  \left(  \nabla F(x_k)^* y_k + w_k \right)$, such that $t_k \rightarrow t$ and $y_k\rightarrow y$.
When $X$ is regular around $\bar{x}$, $T_X(x) $ is convex for all $x\in X$ around $\bar{x}$. With nonexpansive property of ${\rm proj}_{T_X(x_k)}$ for sufficiently large $k$, we have
\begin{equation*}
\begin{aligned}
     & \left\| {\rm proj}_{T_X(x_k)}  \left(  \nabla F(x_k)^* y + w_k \right) -t \right\|   \\
  \leq &  \left\| {\rm proj}_{T_X(x_k)}  \left(  \nabla F(x_k)^* y + w_k \right) -  {\rm proj}_{T_X(x_k)}  \left(  \nabla F(x_k)^* y_k + w_k \right) \right\|  \\ 
  & \hskip5cm +\left\| {\rm proj}_{T_X(x_k)}  \left(  \nabla F(x_k)^* y_k + w_k \right) -t \right\| \\
 \leq  &  \left\| \nabla F(x_k)^*\left( y - y_k\right)  \right\|  + \left\| {\rm proj}_{T_X(x_k)}  \left(  \nabla F(x_k)^* y_k + w_k \right) -t \right\|.
\end{aligned}
\end{equation*}
As $\left\| \nabla F(x_k)^*\left( y - y_k\right)  \right\| $ and $\left\| {\rm proj}_{T_X(x_k)}  \left(  \nabla F(x_k)^* y_k + w_k \right) -t \right\|$ both tend to $0$ when $k\rightarrow \infty$, we have $ {\rm proj}_{T_X(x_k)}  \left(  \nabla F(x_k)^* y + w_k \right) \rightarrow t$ as well. Therefore we have
\begin{align*}
    & \limsup_{\tiny x\xrightarrow[]{X}\bar{x}, y'\xrightarrow[]{} y} \left\{   {\rm proj}_{T_X(x)}\left( \nabla F(x)^*y' + w\right), \ w \in N_X(x) \right\}  \\ 
    \subseteq  & \  \limsup_{\tiny x\xrightarrow[]{X}\bar{x}} \left\{   {\rm proj}_{T_X(x)}\left( \nabla F(x)^*y + w\right), \ w \in N_X(x) \right\}. 
\end{align*}
Given that the inclusion in reverse is obvious by taking $y_k:=y$, we arrive at (\ref{ProjCodeofSmoothFunc2}).

If furthermore $X$ is a smooth manifold around $\bar{x}$, $T_X(x)$ and $N_X(x)$ are linear subspaces orthogonally complementary to each other for any $x\in X$ being sufficiently close to $\bar{x}$ (\cite[Example 6.8]{VaAn}).  By \cite[Theorem 12.22]{VaAn},
$$ {\rm proj}_{T_X(x)} \left( \nabla F(x)^* y   + w \right) = {\rm proj}_{T_X(x)} \left( \nabla F(x)^* y  \right) ,\ \forall w\in N_X(x).$$
Then  the fixed-point expression (\ref{ProjCodeofSmoothFunc-fixed}) is obtained along with the continuity of $T_X(x)$ and $N_X(x)$.
\end{proof}

\section{Lipschitz-like property relative to a smooth manifold}\label{sect:LipSM}
In the last section, we derived a fixed-point expression for projectional coderivative of a mapping relative to a smooth manifold (see Proposition \ref{Prop:SM-ProjCode}). Considering that the generalized Mordukhovich criterion (\Cref{thm-genMordcri}) requires that $X$ is a closed and convex set, in this section, we extend the criterion to the setting of a smooth manifold.

First, we give the sufficient and necessary conditions respectively for $S$ to be Lipschitz-like relative to a smooth manifold. Recall that the Lipschitz-like property relative to a set is given in Definition \ref{def-lip-like}.
 \begin{lemma}\label{lm:AprtSM-Nec}[Necessity, \cite[Theorem 2.1]{Meng2020}]
  Consider a mapping $S: \mathbb{R}^n\rightrightarrows  \mathbb{R}^m$, $\bar{x}\in X\subset \mathbb{R}^n$ where $X$ is a smooth manifold around $\bar{x}$, $\bar{u}\in S(\bar{x})$ and $\kappa\geq 0$. If $S$ has the Lipschitz-like property relative to $X$ at $\bar{x}$ for $\bar{u}$ with constant $\kappa$,   then  the condition
  \begin{equation*}
    \|{\rm proj}_{T_X(x)}(x^*)\|\leq \kappa \|u^*\|\quad\forall x^*\in \widehat{D}^*S\vert_X(x\mid u)(u^*)
  \end{equation*}
  holds  for all $(x,u)$ close enough to $(\bar{x}, \bar{u})$ in $\gph S\vert_X$.
 \end{lemma}

 The necessary condition is a direct application of \cite[Theorem 2.1]{Meng2020}. For the sufficient condition, some efforts need to be made to change the set from $\cl \pos (X - x)$ to $T_X(x)$. We give the proof similar to the one in \cite[Theorem 2.2]{Meng2020} employing the property of smooth manifold, Proposition \ref{Prop:SM-basic} (c). 
\begin{lemma}[Sufficiency]\label{lm:APrtSM-Suf}
Consider a mapping $S: \mathbb{R}^n\rightrightarrows  \mathbb{R}^m$, $\bar{x}\in X\subset \mathbb{R}^n$ where $X$ is a smooth manifold around $\bar{x}$,  $\bar{u}\in S(\bar{x})$ and $\tilde{\kappa}>\kappa>0$. Suppose that $\gph S$ is locally closed at $(\bar{x}, \bar{u})$.   If the  condition
\begin{equation}\label{ineq:projCode-Suf}
\|{\rm proj}_{T_X(x)}(x^*)\|\leq \kappa \|u^*\|,\quad \forall x^*\in D^*S\vert_X(x\mid u)(u^*)
\end{equation}
holds for all  $(x,u)$ close enough to $(\bar{x}, \bar{u})$ in $\gph S\vert_X$, then $S$ has the Lipschitz-like property relative to $X$ at $\bar{x}$ for $\bar{u}$ with  constant $\tilde{\kappa}$.
\end{lemma}
\begin{proof}
The proof is similar to that in \cite[Theorem 2.2]{Meng2020} but differs in using $T_X(x)$ instead of $\cl \pos (X-x)$ by virtue of \Cref{Prop:SM-basic} (c).

Let $0<\varepsilon'<\frac{\tilde{\kappa}-\kappa}{\tilde{\kappa}+ \kappa}$. Then by Proposition \ref{Prop:SM-basic} (c), there is some $\delta>0$ such that the following holds for all $x',\tilde{x}\in X\cap \mathbb{B}_{\delta}(\bar{x})$ with $x'\not=\tilde{x}$:
\begin{equation}\label{ineq:projClose}
\displaystyle\left\|\;\frac{x'-\tilde{x}}{\|x'-\tilde{x}\|}-{\rm proj}_{T_X(\tilde{x})}\left( \frac{x'-\tilde{x}}{\|x'-\tilde{x}\|} \right) \;\right\|\leq \varepsilon'.
\end{equation}
Let $0<\varepsilon<\displaystyle\min\left\{\frac{\tilde{\kappa}-\kappa-(\kappa+\tilde{\kappa})\varepsilon'}{4\tilde{\kappa}(1+\varepsilon')}, \frac13\delta\right\}$.  The selection of $x', x''$ and obtaining $(\tilde{x}, \tilde{u}) \in \gph S\vert_X $ via the Ekeland's variational principle and $(x^*, u^*) \in N_{\gph S\vert_X} (\tilde{x}, \tilde{u})$ via Fermat's rule are the same as in the proof of \cite[Theorem 2.2]{Meng2020}.

It remains to construct a vector $w \in T_X(\tilde{x})$ and formulate a contradiction. Given $\tilde{x},x'\in X\cap \mathbb{B}_{3\varepsilon}(\bar{x})\subset  X\cap \mathbb{B}_{\delta}(\bar{x})$ with $\tilde{x}\not=x'$, let
$$ w^*:=\frac{x'-\tilde{x}}{\|x'-\tilde{x}\|}\quad \mbox{and}\quad w:={\rm proj}_{T_X(\tilde{x})}\frac{x'-\tilde{x}}{\|x'-\tilde{x}\|}\in T_X(\tilde{x}).$$
It follows from (\ref{ineq:projClose}) that
$
\|w^*-w\|\leq \varepsilon'.
$
Then by similar argument in \cite[Theorem 2.2]{Meng2020} and the choice of $\varepsilon$, we have
$$\langle x^*, w\rangle-\kappa \|u^*\|=  \displaystyle \langle x^*, w^*\rangle-\kappa \|u^*\|+\langle x^*, w-w^*\rangle  >0 $$
and
$$
\displaystyle{\rm proj}_{T_X(\tilde{x})}(x^*)=\max_{\tilde{w}\in T_X(\tilde{x})\cap \Sph}\langle x^*, \tilde{w}\rangle\geq \langle x^*, w \rangle>\kappa \|u^*\|,
$$ which thereby joins the contradiction argument in \cite[Theorem 2.2]{Meng2020}. 
\end{proof}
Next we present the characterization of the Lipschitz-like property relative to a smooth manifold in full.  In \cite[Proposition 18]{Daniilidis2011}, they showed that the condition (d) in the following theorem provided the sufficiency. We improve this result with necessity implemented. {\color{black} Recall that the notation $\lvert\cdot\rvert^+$ is the outer norm of a set-valued mapping (see \eqref{eq:def-outernorm}).}

\begin{theorem}[Lipschitz-like property relative to a smooth manifold]\label{thm:AprtSM-criterion}
 Consider a mapping $S: \mathbb{R}^n\rightrightarrows  \mathbb{R}^m$, $\bar{x}\in X\subset \mathbb{R}^n$ where $X$ is a smooth manifold around $\bar{x}$, and $\bar{u}\in S(\bar{x})$. Suppose that $\gph S$ is locally closed at $(\bar{x}, \bar{u})$.  The following properties are equivalent:
 \begin{description}
   \item[(a)] $S$ has the Lipschitz-like property relative to $X$ at $\bar{x}$ for $\bar{u}$.
             \item [(b)] ${\rm proj}_{T_X(\bar{x})}D^*S\vert_X(\bar{x}\mid\bar{u})(0)=\{0\}$.
          \item[(c)] $\lvert \, {\rm proj}_{T_X(\bar{x})}D^* S\vert_X(\bar{x}\mid\bar{u})\,\rvert^+<+\infty$.
      \item[(d)] $D^*S\vert_X(\bar{x}\mid\bar{u})(0)\cap T_X(\bar{x})=0$.
          \item[(e)] $D^*S\vert_X(\bar{x}\mid\bar{u})(0)= N_X(\bar{x})$.
          \item[(f)] $D^*_XS(\bar{x}\mid \bar{u})(0)=\{0\}$.
         \end{description}
Furthermore, we have
  \begin{equation}\label{eqn:AprtSM-LipMod}
   {\rm lip}_X S(\bar{x}\mid\bar{u})=\lvert\,{\rm proj}_{T_X(\bar{x})}D^* S\vert_X(\bar{x}\mid\bar{u})\,\rvert^+.
  \end{equation}
\end{theorem}
\begin{proof} 
It is clear to see that
 \begin{equation}\label{incl-AprtSM-triv}
  D^*S\vert_X(\bar{x}\mid\bar{u})(0)\supset \widehat{D}^*S\vert_X(\bar{x}\mid\bar{u})(0)\supset \widehat{D}^*S(\bar{x}\mid\bar{u})(0)+ N_X(\bar{x})\supset N_X(\bar{x}),
 \end{equation}
 and that the mapping ${\rm proj}_{T_X(\bar{x})}D^* S\vert_X(\bar{x}\mid\bar{u})$ is outer semicontinuous (see \Cref{coro:Osc-ProjCode}) and positively homogeneous.  Then the equivalence of (b) and (c)  follows immediately from \cite[Proposition 9.23]{VaAn}. The equivalences among (b), (d) and (f) follows readily from Proposition \ref{Prop:SM-basic} (c).
 In view of (\ref{incl-AprtSM-triv}), we get the equivalence of (b) and (e).
It remains to prove the equivalence of (a) and (b). 

[(a) $\Longrightarrow$ (c)] Assuming (a), we will show (c) by proving the inequality
 \begin{equation}\label{ineq:AprtSM-LipModUp}
   \lvert\,{\rm proj}_{T_X(\bar{x})}D^* S\vert_X(\bar{x}\mid\bar{u})\,\rvert^+\leq {\rm lip}_X S(\bar{x}\mid\bar{u}).
 \end{equation}
Choose any $\kappa\in ({\rm lip}_X S(\bar{x}\mid\bar{u}), +\infty)$. Then $S$ has the Lipschitz-like property relative to $X$ at $\bar{x}$ for $\bar{u}$ with constant $\kappa$. Let $(u^*, v^*)$ be  given arbitrarily such that $v^*\in {\rm proj}_{T_X(\bar{x})}D^* S\vert_X(\bar{x}\mid\bar{u})(u^*)$. Then there is some $x^*\in D^* S\vert_X(\bar{x}\mid\bar{u})(u^*)$ such that
 \begin{equation}\label{eqn-AprtSM-Cal1}
  v^*={\rm proj}_{T_X(\bar{x})}(x^*).
 \end{equation}
By the definition of the limiting coderivatives, there are some $(x_k, u_k)\rightarrow (\bar{x}, \bar{u})$ with $(x_k, u_k)\in \gph S\vert_X$ and $x^*_k\in \widehat{D}^*S\vert_X(x_k\mid u_k)(u_k^*)$ such that $(x^*_k, -u^*_k)\rightarrow (x^*, -u^*)$. By Lemma \ref{lm:AprtSM-Nec}, there exists some positive integer $k'$ such that
 \begin{equation}\label{eqn-AprtSM-Cal2}
  \|{\rm proj}_{T_X(x_k)}(x_k^*)\|\leq \kappa \|u^*_k\|\quad\forall k\geq k'.
 \end{equation}
 Since $X$ is a smooth manifold around $\bar{x}$, we have
 \begin{equation}\label{eqn-AprtSM-Cal3}
  {\rm proj}_{T_X(x_k)}(x_k^*)\to {\rm proj}_{T_X(\bar{x})}(x^*).
 \end{equation}
In view of (\ref{eqn-AprtSM-Cal1}-\ref{eqn-AprtSM-Cal3}), we have $\|v^*\|\leq \kappa \|u^*\|$ and hence
 \[
  \lvert\,{\rm proj}_{T_X(\bar{x})}D^* S\vert_X(\bar{x}\mid\bar{u})\,\rvert^+\leq \kappa.
 \]
As $\kappa\in ({\rm lip}_X S(\bar{x}\mid\bar{u}), +\infty)$ is chosen arbitrarily, we get (\ref{ineq:AprtSM-LipModUp}) immediately.

[(c) $\Longrightarrow$ (a)] Assuming (c), we will show (a) by proving the inequality 
  \[
    {\rm lip}_X S(\bar{x}\mid\bar{u})\leq \lvert\,{\rm proj}_{T_X(\bar{x})}D^* S\vert_X(\bar{x}\mid\bar{u})\,\rvert^+,
  \]
 from which the equality (\ref{eqn:AprtSM-LipMod})  follows as the inequality in the other direction has been proved earlier.

  Suppose by contradiction that $\lvert\,{\rm proj}_{T_X(\bar{x})}D^* S\vert_X(\bar{x}\mid\bar{u})\,\rvert^+<{\rm lip}_X S(\bar{x}\mid\bar{u})$. Choose any $\kappa', \kappa''$ as
  $$    \kappa'\in   \left(\lvert\,{\rm proj}_{T_X(\bar{x})}D^* S\vert_X(\bar{x}\mid\bar{u})\,\rvert^+,\, {\rm lip}_X S(\bar{x}\mid\bar{u})\right) ,    \kappa''\in  \left(\lvert\,{\rm proj}_{T_X(\bar{x})}D^* S\vert_X(\bar{x}\mid\bar{u})\,\rvert^+,\, \kappa'\right).$$
   Clearly, $S$ fails to have  the Lipschitz-like property relative to $X$ at $\bar{x}$ for $\bar{u}$ with constant $\kappa'$. By Lemma \ref{lm:APrtSM-Suf},
     there exist  some sequences
$(x_k,u_k)\to(\bar{x},\bar{u})$ with $(x_k,u_k)\in \gph S\vert_X$ and some $x^*_k\in D^*S\vert_X(x_k\mid u_k)(u^*_k)$ such that
 $\|v^*_k\|> \kappa'' \|u^*_k\|,\ \forall k, $
where $v^*_k:={\rm proj}_{T_{X}(x_k)}(x^*_k)$.
By Proposition \ref{Prop:SM-ProjCode}, we have $$ v^*_k\in D^*S\vert_X(x_k\mid u_k)(u^*_k)\cap T_X(x_k),\ \forall k.$$
 Clearly, we have $v^*_k\not=0$ for all $k$.
 By taking a subsequence if necessary, we assume that there is some $v^*\in T_X(\bar{x})$ with $\|v^*\|=1$ such that $ \frac{v_k^*}{\|v^*_k\|}\to v^*$.
  As we have $ \frac{\|u^*_k\|}{\|v^*_k\|}< \frac{1}{\kappa''}$ for any $k$, 
by taking a subsequence if necessary again, we assume that there is some $u^*$ with $\kappa''\|u^*\|\leq 1$ such that $\frac{u^*_k}{\|v^*_k\|}\to u^*$. 
 Thus, we have $v^*\in D^*S\vert_X(\bar{x}\mid \bar{u})(u^*)$ and $\|v^*\|\geq \kappa'' \|u^*\|$.
 So we have
\[
\begin{array}{lll}
\lvert\, {\rm proj}_{T_X(\bar{x})}D^* S\vert_X(\bar{x}\mid\bar{u})\,\rvert^+&:=&\displaystyle \sup_{\tilde{u}^*\in \mathbb{B}} ~~ \sup_{\tilde{x}^*\in D^*S\vert_X(\bar{x}\mid\bar{u})(\tilde{u}^*)}\|{\rm proj}_{T_X(\bar{x})}(\tilde{x}^*)\|\\[0.5cm]
&\geq & \|{\rm proj}_{T_X(\bar{x})}(\kappa'' v^*)\| = \kappa''\|v^*\| =\kappa'' 
\end{array}
\]
contradicting to the setting that $\kappa''\in \left(\lvert\,{\rm proj}_{T_X(\bar{x})}D^* S\vert_X(\bar{x}\mid\bar{u})\,\rvert^+,\, \kappa'\right)$.  This completes the proof. \end{proof}

\section{Chain rules for projectional coderivatives}\label{sect:chainrules}

To broaden the scope of application of the projectional coderivative to various systems, one important thing would be developing the corresponding calculus rules, which is also the main goal of the coming two sections. In this section, we focus on the chain rules, i.e., the calculation of $D^*_X S$ with $S = S_2 \circ S_1$. Unlike the chain rule for coderivatives (see \cite[Theorem 10.37]{VaAn}), the one for projectional coderivatives comes with stricter assumptions as it involves projection. 
\begin{theorem}[Projectional coderivative chain rule]\label{Thm-ProjCode-ChainRule}
  Suppose $S = S_2 \circ S_1$ for mappings $S_1: \mathbb{R}^n \rightrightarrows \mathbb{R}^p$ and $S_2: \mathbb{R}^p \rightrightarrows \mathbb{R}^m$. Let $X \subseteq \mathbb{R}^n$ be a closed set with $\bar{x} \in X$. Here $S_1$ is outer semicontinuous relative to $X$ and $S_2$ is outer semicontinuous. For a pair $(\bar{x}, \bar{u}) \in \gph S\vert_X = \gph (S_2 \circ S_1\vert_X)$,  assume:
  \begin{description}
    \item[(a)] the mapping $(x,u) \mapsto S_1 \vert_X (x)  \cap S^{-1}_2 (u)$ is locally bounded at $(\bar{x},\bar{u})$, or equivalently, the mapping $(x, u) \mapsto S_1(x)\cap S^{-1}_2(u)$ is locally bounded relative to $X\times \mathbb{R}^m$ at $(\bar{x},\bar{u})$ (this being true in particular if either $S_1$ is locally bounded relative to $X$ at $\bar{x}$ or $S^{-1}_2$ is locally bounded at $\bar{u}$). In this way, $S_2 \circ S_1\vert_X$ is outer semicontinuous (see \cite[Proposition 5.52 (b)]{VaAn}).
    \item[(b)] $D^*S_2 (\bar{w} \mid \bar{u})(0) \cap D^*_X S_1 (\bar{x}\mid\bar{w})^{-1} (0) = \{0\}$ holds for any $\bar{w} \in S_1\vert_X(\bar{x})\cap S^{-1}_2(\bar{u})$ (this being true in particular if $S_2$ has Lipschitz-like property at $\bar{w}$ for $\bar{u}$).
  \end{description}
  Then $\gph S\vert_X$ is locally closed around $(\bar{x},\bar{u})$ and
  \begin{equation}\label{projcodeinclforchainrule}
    D^*_X S(\bar{x} \mid\bar{u}) \subset \bigcup_{\bar{w} \in S_1\vert_X(\bar{x})\cap S^{-1}_2 (\bar{u})} D^*_X S_1 (\bar{x}\mid\bar{w}) \circ D^*S_2 (\bar{w}\mid\bar{u}).
  \end{equation}
  Besides, if (a) and (b) hold, $S_1\vert_X$ and $S_2$ are graph-convex and $X$ is a smooth manifold around $\bar{x}$, then $S\vert_X$ is graph-convex as well and we can obtain an equation:
  \begin{equation}\label{projcodeequaforchainrule}
    D^*_X S(\bar{x} \mid\bar{u}) =  D^*_X S_1 (\bar{x}\mid\bar{w}) \circ D^*S_2 (\bar{w}\mid\bar{u}), \forall  \bar{w} \in S_1\vert_X(\bar{x})\cap S^{-1}_2(\bar{u}).
  \end{equation}
\end{theorem}

\begin{proof}
  By Lemma \ref{lm:ProjCodeInverseAt0-Incl}, we have that the constraint qualification (b) also indicates the constraint qualification below:
  \begin{equation}\label{CQforChainRule-Code}
    D^*S_2 (\bar{w} \mid\bar{u})(0) \cap D^* S_1\vert_X (\bar{x}\mid\bar{w})^{-1} (0) = \{0\}.
  \end{equation}
  Let $C = \{(x,w,u) \mid (x,w) \in \gph S_1\vert_X, \ (w,u) \in \gph S_2\}$ and  $G: (x,w,u) \mapsto (x,u)$. Then $\gph S\vert_X  = G(C)$. With assumption (a), we can obtain $\varepsilon > 0$ such that $G^{-1}( \mathcal{N}_\varepsilon (\bar{x},\bar{u}) )\cap C$ is bounded.
  Then by \cite[Theorem 6.43]{VaAn} on $\gph S\vert_X $ at $(\bar{x},\bar{u})$, we have that $\gph S\vert_X$ is locally closed at $(\bar{x}, \bar{u})$ and
  \begin{align}
    N_{\gph S\vert_X}(\bar{x}, \bar{u}) \subset &  \bigcup_{(\bar{x},\bar{w},\bar{u})\in G^{-1}(\bar{x},\bar{u})\cap C} \left\{ ( v, -y) \ \mid \ \nabla G(\bar{x},\bar{w},\bar{u})^* (v,-y) \in N_C(\bar{x},\bar{w},\bar{u})\right\}  \nonumber  \\
    =  &   \bigcup_{\bar{w} \in S_1\vert_X (\bar{x}) \cap S^{-1}_2 (\bar{u})} \left\{ ( v, -y) \ \mid \ (v, 0, -y) \in N_C(\bar{x},\bar{w},\bar{u})\right\} \label{eqn:pf-CR-1}.
  \end{align}
 \par
  Next we try to obtain the expression for $N_C(\bar{x},\bar{w},\bar{u})$. Let $D = \gph S_1\vert_X \times \gph S_2$. For $F:(x,w,u) \mapsto (x,w,w,u)$, we have $C = F^{-1}(D)$. Here the definition of $F$ ensures the component $w$ of $(x,w,w,u)$ in $D$ belongs to $S_1 \vert_X (x) \cap S^{-1}_2 (u)$.  Here we apply \cite[Theorem 6.14]{VaAn} on $C = F^{-1}(D)$. The constraint qualification requires that:
  $$\forall  q \in N_D(F(\bar{x},\bar{w},\bar{u})) \text{ with } -\nabla F(\bar{x},\bar{w},\bar{u})^*  q = 0 \Longrightarrow q =0,$$
  which is
  \begin{equation*}
      \begin{cases}
         ( q_1,  q_2) \in N_{\gph S_1\vert_X}(\bar{x}, \bar{w}) \\
         ( q_3,  q_4) \in N_{\gph S_2}(\bar{w},\bar{u}) \\ 
         (q_1, q_2 + q_3, q_4) =0
      \end{cases} \Longrightarrow q_1, q_2, q_3, q_4 = 0 
  \end{equation*}
due to the product form of $D = \gph S_1\vert_X \times \gph S_2$ (see \cite[Proposition 6.41]{VaAn}).
  By expressing in coderivatives, it becomes
  \begin{equation*}
    0 \in D^*S_1\vert_X(\bar{x}\mid\bar{w})(q_3) ,\  q_3 \in  D^*S_2(\bar{w}\mid\bar{u})(0) \Longrightarrow q_3 =0  \text{ for all } \bar{w}\in S_1\vert_X (\bar{x}) \cap S^{-1}_2 (\bar{u}),
  \end{equation*}
  which can be reformulated as in (\ref{CQforChainRule-Code}).
  Then we can have the inclusion:
  \begin{align*}
     N_C(\bar{x},\bar{w},\bar{u}) \subset & \left\{ \nabla F(\bar{x},\bar{w},\bar{u})^* q \  \mid  \ q \in N_D(\bar{x},\bar{w},\bar{w}, \bar{u} )  \right\} \\ 
    =  &  \left\{(q_1, q_2+ q_3,q_4) \  \mid  \  (q_1, q_2) \in N_{\gph S_1\vert_X}(\bar{x}, \bar{w}), (q_3, q_4) \in N_{\gph S_2}(\bar{w},\bar{u}) \right\}.
  \end{align*}
   Together with \eqref{eqn:pf-CR-1} we have
   \begin{multline*}
         N_{\gph S\vert_X}(\bar{x}, \bar{u}) \subseteq  \bigcup_{\bar{w} \in S_1\vert_X (\bar{x}) \cap S^{-1}_2 (\bar{u})}  \bigg\{ ( x^*, -u^*)  \ \big\vert \ \exists w^* \text{ s.t. }  \\
        x^* \in D^*S_1\vert_X (\bar{x}\mid\bar{w})(w^*) , w^* \in D^*S_2 (\bar{w}\mid\bar{u}) (u^*)\bigg\}.
   \end{multline*}
   Next we prove that the constraint qualification (\ref{CQforChainRule-Code}) also holds for all $(x,u)$ in $\gph S\vert_X $ sufficiently near to $(\bar{x}, \bar{u})$ by contradiction.
   Suppose there exist sequences $(x_k, u_k)\xrightarrow{\gph S\vert_X} (\bar{x}, \bar{u})$, $w_k \in S_1\vert_X (x_k ) \cap S_2^{-1} (u_k) $ , 
   and $w^*_k  \in D^*S_2 (w_k \mid u_k)(0) \cap D^*S_1\vert_X (x_k\mid w_k)^{-1} (0)$ (which is a cone) such that $w^*_k \neq 0$. 
   Without loss of generality we assume $\|w^*_k\|=1$. Note that under assumption (a), $w_k \rightarrow \bar{w} \in S_1\vert_X (\bar{x} ) \cap S_2^{-1} (\bar{u})$. 
   By outer semicontinuity of normal cone mappings, $w^*_k $ must converge to some $ w^* \in D^*S_2 (\bar{w} \mid\bar{u})(0) \cap D^*S_1\vert_X (\bar{x}\mid\bar{w})^{-1} (0) $ with $ \|w^*\| =1$, which contradicts (\ref{CQforChainRule-Code}).
  As the assumption (a) and (\ref{CQforChainRule-Code}) hold for all $(x,u)$ in $\gph S\vert_X $ around $(\bar{x}, \bar{u})$, the inclusion can be obtained:
  \begin{multline}\label{incl3inpf}
       N_{\gph S\vert_X}(x,u) \subset \bigcup_{w \in S_1\vert_X (x) \cap S^{-1}_2 (u)} \bigg\{ ( x^*, -u^*) \   \big\vert \ \exists w^* \text{ s.t. } \\
       x^* \in D^*S_1\vert_X (x\mid w)(w^*) , w^* \in D^*S_2 (w\mid u) (u^*)\bigg\} 
  \end{multline}
  for all $(x,u)$ in $\gph S\vert_X$ around $(\bar{x}, \bar{u})$.
  Given the upper estimate of $N_{\gph S\vert_X}(x,u)$ in (\ref{incl3inpf}), we now proceed to exploring the estimate of projectional coderivative $D^*_X S(\bar{x},\bar{u})$. Let $t^* \in  D^*_X S(\bar{x}\mid \bar{u})(u^*) $, then there are sequences $(x_k, u_k)\xrightarrow{\gph S\vert_X} (\bar{x}, \bar{u})$ and $(x^*_k, -u^*_k) \in N_{\gph S\vert_X}(x_k, u_k) $  with $t^*_k \in {\rm proj}_{T_X(x_k)} (x^*_k )$, such that $t^*_k\rightarrow t^*$ and $u^*_k \rightarrow u^*$. By (\ref{incl3inpf}), there exist $ w_k \in S_1\vert_X(x_k) \cap S^{-1}_2 (u_k)$ and $w^*_k$ such that $(x^*_k, -w^*_k) \in N_{\gph S_1\vert_X} (x_k,w_k)$ and $(w^*_k,-u^*_k) \in N_{\gph S_2}(w_k,u_k)$.

Given $w_k \in S_1\vert_X(x_k) \cap S^{-1}_2 (u_k)$, the outer semicontinuity of $S_1\vert_X$ and $S^{-1}_2$ and  local boundedness of the mapping $(x,u) \mapsto S_1\vert_X (x)  \cap S^{-1}_2 (u)$ around $(\bar{x}, \bar{u})$, $\{w_k\}$ must converge to some $\bar{w} \in S_1\vert_X(\bar{x}) \cap S^{-1}_2 (\bar{u})$ (taking a subsequence if necessary).
For $(w^*_k,-u^*_k) \in N_{\gph S_2}(w_k,u_k)$ and outer semicontinuity of $(w,u) \mapsto N_{\gph S_2 }(w,u)$ at $(\bar{w},\bar{u})$, we have either $w^*_k \rightarrow w^*$  or $\lambda_k w^*_k \rightarrow w^*$ with $\lambda_k \searrow 0$.
For the first case we have $w^* \in D^*S_2 (\bar{w} \mid \bar{u})(u^*)$. Given $(x^*_k, -w^*_k) \in N_{\gph S_1\vert_X} (x_k,w_k)$ and $t^*_k \in {\rm proj}_{T_X(x_k)} (x^*_k )$ with $t^*_k \rightarrow t^*$, then $t^* \in D^*_X S_1(\bar{x}\mid\bar{w})(w^*)$. Thus $t^* \in D^*_X S_1 (\bar{x}\mid\bar{w}) \circ D^*S_2 (\bar{w} \mid \bar{u})(u^*)$ with $\bar{w} \in S_1\vert_X (\bar{x})\cap S^{-1}_2(\bar{u})$.

For the second case, without loss of generality we can assume $\|w^*\| =1$. Under the conic nature, $\lambda_k w^*_k \in D^*S_2(w_k\mid u_k)(\lambda_k u^*_k)$. Given $\{u^*_k\}$ is bounded with $u^*_k\rightarrow u^*$, then $\lambda_k u^*_k \rightarrow 0$ and we have $w^* \in D^*S_2 (\bar{w} \mid \bar{u})(0) $. Similarly we have $(\lambda_k x^*_k, -\lambda_k w^*_k) \in N_{\gph S_1\vert_X} (x_k,w_k)$.
As $T_X(x)$ is a nonempty closed cone for any $x\in X$ around $\bar{x}$, $\lambda_k t^*_k \in {\rm proj}_{T_{X}(x_k)}(\lambda_k x^*_k) $ and $\lambda_k t^*_k \rightarrow 0$. That is, $0 \in D^*_X S_1 (\bar{x} \mid \bar{w})(w^*)$. Thus we have $w^* \in D^*S_2 (\bar{w} \mid\bar{u})(0) \cap D^*_X S_1 (\bar{x}\mid\bar{w})^{-1} (0) = \{0\}$ with $\|w^*\| =1$, which contradicts the assumption (b).
Therefore the case $\{\lambda_k w^*_k\} \rightarrow w^*$ can be abandoned and the inclusion (\ref{projcodeinclforchainrule}) is thus proved.

Note that by definition, ${\rm proj}_{T_X(\bar{x})}D^*S\vert_X(\bar{x}\mid \bar{u}) \subseteq D^*_X S(\bar{x} \mid \bar{u})$.
When $\gph S_1\vert_X$   and $\gph S_2$ are convex, the inclusion (\ref{incl3inpf}) becomes an equation for every $w \in S_1\vert_X (x) \cap S^{-1}_2 (u)$  and the union becomes superfluous (see \cite[Theorem 10.37]{VaAn}).
Then we have
  \begin{equation*}\label{projcodelowerinclforchainrule}
    {\rm proj}_{T_X(\bar{x})}D^*S_1\vert_X(\bar{x} \mid \bar{w}) \circ D^*S_2 (\bar{w} \mid \bar{u}) \subseteq D^*_X S(\bar{x} \mid \bar{u}),\ \forall  \bar{w} \in S_1\vert_X(\bar{x})\cap S^{-1}_2(\bar{u}).
  \end{equation*} and that $\gph S\vert_X$ is convex as well.
Besides, when $X$ is a smooth manifold around $\bar{x}$, by Proposition \ref{Prop:SM-ProjCode}, 
$$   D^*_X S_1 (\bar{x}\mid\bar{w}) \circ D^*S_2 (\bar{w}\mid\bar{u}) =  {\rm proj}_{T_X(\bar{x})}D^*S_1\vert_X(\bar{x}\mid \bar{w}) \circ D^*S_2 (\bar{w}\mid\bar{u}) .$$
Then the equation (\ref{projcodeequaforchainrule}) is obtained.
\end{proof}

Based on Theorem \ref{Thm-ProjCode-ChainRule} above, similar to \cite[Exercise 10.39, Theorem 10.40]{VaAn}, we give the following two subsequent results when one of the mapping in the composition is single-valued. When the outer layer is single-valued, we can apply Theorem \ref{Thm-ProjCode-ChainRule} directly. 

\begin{corollary}[Outer composition with a single-valued function]\label{CoroOuterComposite-FS}
  Let  $X$ be a closed set in $\mathbb{R}^n$ and $S = F\circ S_0$ for a mapping $S_0 : \mathbb{R}^n \rightrightarrows \mathbb{R}^p$ being outer semicontinuous relative to $X$ and a single-valued function $F: \mathbb{R}^p \rightarrow \mathbb{R}^m$. Let $\bar{u} \in S\vert_X(\bar{x})$ and suppose $F$ is strictly continuous at every $\bar{w} \in S_0(\bar{x})$. Suppose also that the mapping $(x,u)\mapsto S_0\vert_X(x) \cap F^{-1}(u)$ is locally bounded at $(\bar{x},\bar{u})$. Then
  \begin{equation*} 
     D^*_X S(\bar{x}\mid\bar{u}) \subseteq \bigcup_{\bar{w} \in S_0(\bar{x}) \cap F^{-1}(\bar{u})} D^*_X S_0(\bar{x}\mid\bar{w}) \circ D^*F(\bar{w}).
  \end{equation*}
  If in addition $S_0\vert_X $ is graph-convex, $X$ is a smooth manifold around $\bar{x}$, and $F$ is linear,
  then
  \begin{equation*}
     D^*_X S(\bar{x}\mid\bar{u}) = D^*_X S_0(\bar{x}\mid\bar{w}) \circ  \nabla F(\bar{w})^* .
  \end{equation*}
\end{corollary}
\begin{proof}
  This result is obtained directly as a special case of Theorem \ref{Thm-ProjCode-ChainRule}.
\end{proof}

Next we provide a simple example for illustration. 
\begin{example}
Let $S_0(x) = \left[ -\sqrt{\lvert x \rvert},\sqrt{\lvert x \rvert} \right], \ F (w) = 2w , \ X = \mathbb{R}_+.$
Then $S\vert_X (x) = F\circ S_0\vert_X (x) = \left[ -2\sqrt{\lvert x \rvert},2\sqrt{\lvert x \rvert} \right]$. 
For $\bar{x} = 0$ and $\bar{u} = 0 \in S\vert_X (\bar{x}) = F\circ S_0\vert_X (\bar{x})$.
Then we have $\bar{w} \in S_0\vert_X(\bar{x}) \cap F^{-1} (\bar{u}) = \{0\}$ and 
$$D^*_X S_0(\bar{x}\mid \bar{w}) (z) = \begin{cases} \mathbb{R}_- & z=0 \\  \emptyset & z \neq 0   \end{cases}, \quad \nabla F(\bar{w})^*y = 2y. $$
Therefore
$$D^*_X S(\bar{x}\mid \bar{u})(y) = D^*_X S_0(\bar{x}\mid \bar{w}) \circ \nabla F(\bar{w})^*y =\begin{cases} \mathbb{R}_- & y=0 \\  \emptyset & y \neq 0   \end{cases} .  $$
\end{example}
  However, when the inner layer of the composition involves a single-valued function, it varies from direct application of Theorem \ref{Thm-ProjCode-ChainRule} in that $F$ is restricted on $X$ rather than defined on the whole space when we try to derive an equation for the projectional coderivative. Before that, we present the expression of the coderivative of $F$ restricted on $X$. 

\begin{lemma}
   Let $X $ be a closed set in $\mathbb{R}^n$ and $ F: \mathbb{R}^n \rightarrow \mathbb{R}^m$ be strictly continuous at $\bar{x}$ relative to $X$. Then for all $z\in \mathbb{R}^m $ we have
   \begin{align}
     \widehat{D}^*F\vert_X (\bar{x}) (z)  & = \widehat{\partial }(zF\vert_X) (\bar{x}) , \label{RegCodeFX} \\
     D^*F\vert_X (\bar{x}) (z)  & = \partial (zF\vert_X) (\bar{x}). \label{CodeFX}
   \end{align}
\end{lemma}
\begin{proof}
  For $v \in \widehat{D}^*F\vert_X (\bar{x}) (z) $, it is equivalent that
  $$\langle v,x-\bar{x} \rangle - \langle z, F\vert_X (x) - F\vert_X(\bar{x}) \rangle \leq o(\|(x, F\vert_X (x)) - (\bar{x}, F\vert_X(\bar{x}) ) \|).$$ As $F$ is strictly continuous at $\bar{x}$ relative to $X$, we can replace $o(\|(x, F\vert_X (x)) - (\bar{x}, F\vert_X(\bar{x}) ) \|)$ with $o(\| x - \bar{x}\|)$, i.e.,
  $$(zF\vert_X ) (x) \geq (zF\vert_X ) (\bar{x}) + \langle v, x- \bar{x} \rangle + o(\|x-\bar{x}\|). $$
  Thus it is equivalent that $v\in \widehat{\partial }(zF\vert_X) (\bar{x})$.
  Given that $F$ is also strictly continuous at $x$ relative to $X$ for $x$ being sufficiently close to $\bar{x}$,  such equation (\ref{RegCodeFX}) also holds for any $x\in X\cap O$ for some $O \in \mathcal{N}(\bar{x})$. Let $v \in D^*F\vert_X (\bar{x}) (z)$, then there exist sequences $x_k \xrightarrow[]{X} \bar{x}$ and $v_k \in \widehat{D}^*F\vert_X (x_k) (z_k)$ with $v_k \rightarrow v$, $z_k \rightarrow z$. By (\ref{RegCodeFX}), $v_k \in \widehat{\partial }(z_k F\vert_X) (x_k)\subseteq \partial(z_k F\vert_X) (x_k) = \partial [ z F\vert_X  + (z_k - z) F\vert_X] (x_k)\subseteq \partial (z F\vert_X )(x_k) + \partial[(z_k - z) F\vert_X] (x_k) $. When $k\rightarrow \infty$, $\partial[(z_k - z) F\vert_X](x_k) \rightarrow \{0\}$ as $z_k \rightarrow z$.
  Given $F$ is strictly continuous at $\bar{x}$ relative to $X$, $zF\vert_X (x_k) \rightarrow zF\vert_X (\bar{x})$ when $x\xrightarrow[]{X}\bar{x}$. Then we have $v_k \rightarrow v \in \partial (z F\vert_X )(\bar{x})$. For the inclusion in reverse for (\ref{CodeFX}), let $v \in  \partial (z F\vert_X )(\bar{x})$. Then there exist sequences $x_k \xrightarrow[]{zF\vert_X} \bar{x}$ and $v_k \in  \widehat{\partial} (z F\vert_X )(x_k)$ such that $ v_k \rightarrow v$. Then it is equivalent that $x\xrightarrow[]{X} \bar{x}$ and $v_k \in \widehat{D}^*F\vert_X (x_k)(z) $ and by definition of normal cone mappings we have $v_k \rightarrow v \in D^*F\vert_X (\bar{x})(z)$.
\end{proof}

\begin{theorem}[Inner composition with a single-valued function]\label{CoroInnerComposite-SF}
  Let $X $ be a closed set in $\mathbb{R}^n$, and $S =  S_0 \circ F$ for an outer semicontinuous mapping $S_0 : \mathbb{R}^p \rightrightarrows \mathbb{R}^m $ and a single-valued mapping $F: \mathbb{R}^n \rightarrow \mathbb{R}^p$. Let $\bar{u} \in S\vert_X (\bar{x})$. If
  \begin{equation}\label{inner-singlevalued-CQ}
    D^*S_0 (F(\bar{x} ) \mid \bar{u} ) (0 ) \cap D^*_X F(\bar{x}) ^{-1}(0) = \{ 0\} ,
  \end{equation}
  then
  \begin{equation}\label{inner-singlevalued-incl}
    D^*_X S(\bar{x} \mid \bar{u}) \subseteq D^*_X F(\bar{x}) \circ D^*S_0 (F(\bar{x})\mid\bar{u}) .
  \end{equation}
  Still under (\ref{inner-singlevalued-CQ}), suppose that $F$ is strictly continuous at $\bar{x}$ relative to $X$, $S_0$ is graphically regular at $F(\bar{x}) $ for $\bar{u}$, the function $zF\vert_X$ is regular at $\bar{x}$ for all $z \in \rg D^*S_0 \left(F(\bar{x})\mid\bar{u}\right)$, and $X$ is a smooth manifold around $\bar{x}$. Then $S\vert_X$ is graphically regular at $\bar{x}$ for $\bar{u}$, and
  \begin{equation}\label{inner-singlevalued-eqn}
    D^*_X S(\bar{x} \mid \bar{u}) =  D^*_X F(\bar{x}) \circ D^*S_0 (F(\bar{x})\mid\bar{u}) .
  \end{equation}
  \end{theorem}
  \begin{proof}
  The inclusion (\ref{inner-singlevalued-incl}) comes from directly applying the chain rule for projectional coderivatives, as $F\vert_X(\cdot) $ is locally bounded and single-valued at $\bar{x}$.
    For the equation, first we prove
    \begin{equation}\label{inner-singlevalued-incl-RNC}
       \widehat{D}^*F\vert_X(\bar{x}) \circ \widehat{D}^*S_0 (F(\bar{x}) \mid \bar{u}) \subseteq  \widehat{D}^*S\vert_X (\bar{x} \mid \bar{u}).
    \end{equation}
    Let $ v\in  \widehat{D}^*F\vert_X(\bar{x})(z)$ and $z\in \widehat{D}^*S_0 (F(\bar{x}) \mid \bar{u})(y)$. By definition we have:
    
    \begin{equation}\label{RNCforFX-1}
    \begin{split}
         \limsup_{\tiny (x,u) \xrightarrow[\neq]{\gph S\vert_X}(\bar{x},\bar{u})} \frac{ \langle v,x-\bar{x} \rangle - \langle z, F\vert_X (x) - F\vert_X(\bar{x}) \rangle}{\|(x- \bar{x}, F\vert_X (x) - F\vert_X(\bar{x}) )\|} \\
         \leq \limsup_{\tiny x\xrightarrow[\neq]{X}\bar{x}} \frac{ \langle v,x-\bar{x} \rangle - \langle z, F\vert_X (x) - F\vert_X(\bar{x}) \rangle}{\|(x- \bar{x}, F\vert_X (x) - F\vert_X(\bar{x}) )\|} \leq 0
    \end{split}
    \end{equation}
    and
    \begin{equation}\label{RNCforS0-1}
    \begin{split}
        \limsup_{\tiny (F\vert_X(x), u)\xrightarrow[\neq]{\gph S_0}(F\vert_X(\bar{x}), \bar{u})} \frac{ \langle z,F\vert_X(x) - F\vert_X(\bar{x}) \rangle - \langle y, u - \bar{u} \rangle}{\| (F\vert_X (x), u)  - ( F\vert_X(\bar{x}),\bar{u} ) \|} \\
        \leq  \limsup_{\tiny (w, u)\xrightarrow[\neq]{\gph S_0}(F\vert_X(\bar{x}), \bar{u})} \frac{ \langle z,w- F\vert_X(\bar{x}) \rangle - \langle y, u - \bar{u} \rangle}{\| (w, u)  - ( F\vert_X(\bar{x}),\bar{u} ) \|} \leq 0 .
    \end{split}
    \end{equation}
    As $F$ is strictly continuous at $\bar{x}$ relative to $X$, $F\vert_X(x) \rightarrow F\vert_X(\bar{x})$ when $x\xrightarrow[]{X} \bar{x}$ and
    $$ \limsup_{\tiny x\xrightarrow[\neq]{X}\bar{x}} \frac{ \|F\vert_X (x) - F\vert_X(\bar{x}) \|}{\|x- \bar{x}\|} < \infty.$$
    Therefore by (\ref{RNCforFX-1}) and (\ref{RNCforS0-1}) we have:
    \begin{align}
       &  \limsup_{\tiny (x,u) \xrightarrow[\neq]{\gph S\vert_X}(\bar{x},\bar{u})} \frac{ \langle v,x-\bar{x} \rangle - \langle z, F\vert_X (x) - F\vert_X(\bar{x}) \rangle}{\|(x- \bar{x}, u-\bar{u})\|} \nonumber \\
      =  &  \limsup_{\tiny (x,u) \xrightarrow[\neq]{\gph S\vert_X}(\bar{x},\bar{u})} \left(  \frac{ \langle v,x-\bar{x} \rangle - \langle z, F\vert_X (x) - F\vert_X(\bar{x}) \rangle}{\|(x- \bar{x}, F\vert_X (x) - F\vert_X(\bar{x}) )\|} \cdot \frac{\|(x- \bar{x}, F\vert_X (x) - F\vert_X(\bar{x}) )\|}{\|(x- \bar{x}, u-\bar{u})\|} \right)  \leq 0 \label{RNCforFX-2}
    \end{align}
    and
    \begin{align}
       &   \limsup_{\tiny (F\vert_X(x), u)\xrightarrow[\neq]{\gph S_0}(F\vert_X(\bar{x}), \bar{u})} \frac{ \langle z,F\vert_X(x) - F\vert_X(\bar{x}) \rangle - \langle y, u - \bar{u} \rangle}{ \|(x- \bar{x}, u-\bar{u})\| }  \nonumber \\
      =  &    \limsup_{\tiny (F\vert_X(x), u)\xrightarrow[\neq]{\gph S_0}(F\vert_X(\bar{x}), \bar{u})} \frac{ \langle z,F\vert_X(x) - F\vert_X(\bar{x}) \rangle - \langle y, u - \bar{u} \rangle}{\| (F\vert_X (x), u)  - ( F\vert_X(\bar{x}),\bar{u} ) \|}\cdot \frac{\| (F\vert_X (x), u)  - ( F\vert_X(\bar{x}),\bar{u} ) \|}{\|(x- \bar{x}, u-\bar{u})\|} \leq 0 . \label{RNCforS0-2}
    \end{align}
    Given that $x\in X$ and $(F\vert_X(x), u) \in \gph S_0$ is equivalently to $(x,u)\in \gph S\vert_X$, combining (\ref{RNCforFX-2}) and (\ref{RNCforS0-2}) we have
    \begin{align*}
       & \limsup_{\tiny (x, u)\xrightarrow[\neq]{\gph S\vert_X}(\bar{x}, \bar{u})} \frac{ \langle v,x-\bar{x} \rangle  - \langle y, u - \bar{u} \rangle}{\|(x, u)  - ( \bar{x},\bar{u} ) \|}  \\
     \leq  &  \limsup_{\tiny (x,u) \xrightarrow[\neq]{\gph S\vert_X}(\bar{x},\bar{u})} \frac{ \langle v,x-\bar{x} \rangle - \langle z, F\vert_X (x) - F\vert_X(\bar{x}) \rangle}{\|(x- \bar{x}, u-\bar{u})\|}  + \\ 
     & \hskip4cm \limsup_{\tiny (x, u)\xrightarrow[\neq]{\gph S\vert_X}(\bar{x}, \bar{u})}  \frac{ \langle z,F\vert_X(x) - F\vert_X(\bar{x}) \rangle - \langle y, u - \bar{u} \rangle}{ \|(x- \bar{x}, u-\bar{u})\| }  \leq  0,
    \end{align*}
    which means $v\in  \widehat{D}^*S\vert_X (\bar{x} \mid \bar{u}) (y)$.
    For the equation part, note that (\ref{inner-singlevalued-CQ}) also indicates
    \begin{equation*} 
      z \in D^*S_0 (F(\bar{x} ) \mid \bar{u} ) (0 ) ,\ 0\in D^*F\vert_X(\bar{x})(z){ = \partial (z F\vert_X) (\bar{x})} \Longrightarrow z =0.
    \end{equation*}
    By \cite[Theorem 10.37]{VaAn}, we have
    \begin{equation}\label{inner-singlevalued-incl-NC}
      D^*S\vert_X (\bar{x} \mid \bar{u}) \subseteq D^*F\vert_X(\bar{x}) \circ D^*S_0 (F(\bar{x}) \mid \bar{u}).
    \end{equation}

    With (\ref{inner-singlevalued-incl-RNC}) and (\ref{inner-singlevalued-incl-NC}), if we assume that $\gph S_0$ is regular at $(F(\bar{x}), \bar{u})$ and the function $zF\vert_X$ is regular at $\bar{x}$ for all $z \in \rg D^*S_0 \left(F(\bar{x})\mid\bar{u}\right)$, we have \begin{equation}\label{inner-singlevalued-eqn-NC}
      D^*S\vert_X (\bar{x} \mid \bar{u}) = D^*F\vert_X(\bar{x}) \circ D^*S_0 (F(\bar{x}) \mid \bar{u}))
    \end{equation}
    and also that $S\vert_X$ is graphically regular at $\bar{x}$ for $\bar{u}$.
    Therefore $${\rm proj}_{T_X(\bar{x}) } D^*F\vert_X(\bar{x}) \circ D^*S_0 (F(\bar{x}) \mid \bar{u})) = {\rm proj}_{T_X(\bar{x}) } D^*S\vert_X (\bar{x} \mid \bar{u}) \subseteq D^*_X S(\bar{x} \mid \bar{u}).$$
    Besides, when $X$ is a smooth manifold around $\bar{x}$,
    \begin{equation}\label{inner-singlevalued-upperincl}
    \begin{aligned}
      D^*_X S(\bar{x} \mid \bar{u}) = {\rm proj}_{T_X(\bar{x}) } D^* S\vert_X (\bar{x}\mid \bar{u}) &  \subseteq {\rm proj}_{T_X(\bar{x})} D^*F\vert_X (\bar{x})\circ D^*S_0 (F(\bar{x}) \mid \bar{u}))\\
      & =   D^*_X F(\bar{x}) \circ D^*S_0 (F(\bar{x})\mid\bar{u})),
    \end{aligned}
    \end{equation}
    where the two equations come from applying Proposition \ref{Prop:SM-ProjCode} to $S$ and $F$ and the inclusion comes from (\ref{inner-singlevalued-incl-NC}).
    Combining these two inclusions (\ref{inner-singlevalued-eqn-NC}) and (\ref{inner-singlevalued-upperincl}), we can obtain (\ref{inner-singlevalued-eqn}).
  \end{proof}

\section{Sum rules for projectional coderivatives}\label{sect:sumrules}
Next we present two sum rules on projectional coderivatives obtained by different methods. Instead of directly applying \Cref{Thm-ProjCode-ChainRule} twice, we develop our sum rules from the sum rule of coderivatives to maintain tighter estimates. The reason behind is the generic asymmetrical nature of the projectional coderivatives: only projecting the first element in $N_{\gph S\vert_X}(x,u)$ to $T_X(x)$. Such an asymmetry also causes the difference when we are imposing the restriction of $X$ to different levels. For the multifunction $S  = S_1 + \dots + S_p$, when restricting $S$ onto $X$ (Sum rule-1), we separate $X$-related expressions from $S_i$ and when restricting each $S_i$ onto $X$ (Sum rule-2), the calculation is performed on $S_i\vert_X$. Readers may choose any one of these sum rules depending on the structure of their problems. 
\begin{theorem}[Sum rule-1]\label{SumRule-2}
Let $S = S_1 + \dots + S_p$ for $S_i : \mathbb{R}^n \rightrightarrows \mathbb{R}^m$ being outer semicontinuous relative to $X$ and let $\bar{x} \in \dom S\cap X$, $\bar{u} \in S\vert_X (\bar{x})$. 
  \begin{itemize}
    \item (boundedness condition): the mapping
     \begin{equation}\label{SumRule-1-BddCond}
       (x, u )\mapsto \left\{ (u_1, \cdots, u_p) \ \bigg\vert \  u_i \in S_i\vert_X (x), \ \forall i= 1, \dots, p,\   \sum_{i=1}^{p}u_i = u\right\}
     \end{equation} is locally bounded at $(\bar{x}, \bar{u})$.
    \item (constraint qualification):
  \begin{equation}\label{CQforSumRule2}
        \left. \begin{matrix}
          v_i \in D^*S_i (\bar{x} \mid u_i) (0),\ u_i \in S_i (\bar{x}), \ \sum_{i=1}^{p} u_i = \bar{u}\\
          0 \in \sum_{i=1}^{p} v_i + N_X(\bar{x})
    \end{matrix} \right\} \Longrightarrow v_i =0 \text{ for } i = 1, \dots ,p
    \end{equation}
  \end{itemize}
    Then $\gph S\vert_X$ is locally closed at $(\bar{x}, \bar{u})$ and one has
  \begin{equation}\label{SumRule-2-Incl1}
    D^*_X S(\bar{x} \mid\bar{u})(y) \subseteq \limsup_{\footnotesize \substack{ (x,u)\xrightarrow[]{\gph S\vert_X}(\bar{x}, \bar{u})\\  y' \xrightarrow[]{}y}}\   \bigcup_{\substack{ u'_i\in S_i(x) \\  \sum_{i=1}^{p} u'_i =u}}  {\rm proj}_{T_X(x)} \left(\sum_{i=1}^{p} D^*S_i (x \mid u'_i)(y') + N_X(x)  \right).
  \end{equation}
  If in addition $X$ is a smooth manifold around $\bar{x}$, the inclusion becomes a fixed-point expression as
  \begin{equation}\label{SumRule-2-Incl2}
     D^*_X S(\bar{x} \mid\bar{u})  \subseteq \bigcup_{\substack{ u_i\in S_i(\bar{x}) \\  \sum_{i=1}^{p} u_i =\bar{u}}}   {\rm proj}_{T_X(\bar{x})} \left( \sum_{i =1 }^{p} D^*S_i (\bar{x}  \mid  u_i ) \right) .
  \end{equation}
  When every $S_i\vert_X$ is graph-convex and $X$ is regular around $\bar{x}$, the union is superfluous and the inclusions \eqref{SumRule-2-Incl1} and \eqref{SumRule-2-Incl2} become equations respectively.
\end{theorem}
\begin{proof}
  As we restrict our scope only to $X$ here, without loss of generality, we can relax the requirement that $S_i$ being outer semicontinuous to being outer semicontinuous relative to the set $X$. Then given (\ref{SumRule-1-BddCond}) and (\ref{CQforSumRule2}) and that (\ref{CQforSumRule2}) holds for points $(x,u) \in \gph S\vert_X$ around $(\bar{x}, \bar{u})$ (as similar proof is given in Theorem \ref{Thm-ProjCode-ChainRule}),  by \cite[Theorem 6.42 and Theorem 10.41]{VaAn}, we have for all $(x,u)\in \gph S\vert_X $ around  $(\bar{x}, \bar{u})$,
    \begin{equation}\label{SumRule-NC-Incl1}
    N_{\gph S\vert_X}(x, u) \subseteq  \bigcup_{\substack{ u'_i\in S_i (x) \\  \sum_{i=1}^{p} u'_i =u}} \left\{ \left(\sum_{i=1}^{p} v_i + w , -y' \right) \Bigg\vert \  v_i \in D^*S_i (x \mid u'_i)(y') , w\in N_X(x)\right\}
  \end{equation}
  and therefore,
  \begin{equation}\label{SumRule-2-projNC-Incl1}
  \begin{aligned}
    & \limsup_{\footnotesize (x,u)\xrightarrow[]{\gph S\vert_X}(\bar{x}, \bar{u})}{\rm proj}_{T_X(x) \times \mathbb{R}^m} N_{\gph S\vert_X}(x, u)  \\
    \subseteq  & \limsup_{\footnotesize (x,u)\xrightarrow[]{\gph S\vert_X}(\bar{x}, \bar{u})} \bigcup_{\footnotesize  \substack{ u'_i\in S_i (x) \\ \sum_{i=1}^{p} u'_i =u}} \left\{ \left( {\rm proj}_{T_X(x)} \left( \sum_{i=1}^{p} v_i +w \right) , -y' \right)\right. \Bigg\vert \ v_i \in D^*S_i (x \mid u'_i)(y'), \\
    & \hskip8cm w\in N_X(x) \Bigg\}.
  \end{aligned}
  \end{equation}
  Given the definition of projectional coderivative and the upper estimate (\ref{SumRule-2-projNC-Incl1}) we have (\ref{SumRule-2-Incl1}).
  Moreover, when $X$ is a smooth manifold around $\bar{x}$, by \Cref{coro:Osc-ProjCode} and the upper estimate \eqref{SumRule-NC-Incl1}, we have 
  \begin{align*}
     D^*_X S(\bar{x} \mid\bar{u})(y)  =  & \   {\rm proj}_{T_X(\bar{x})} D^*S\vert_X(\bar{x}\mid\bar{u})(y)  \\
      \subseteq & \bigcup_{\substack{ u_i\in S_i(\bar{x}) \\  \sum_{i=1}^{p} u_i =\bar{u}}}   \left\{   {\rm proj}_{T_X(\bar{x})}\left( \sum_{i=1}^{p} v_i + w\right) \Bigg\vert \ v_i \in D^*S_i(\bar{x} \mid u_i) (y) , w\in N_X(\bar{x})  \right\}.
  \end{align*}
  By orthogonality between $T_X(\bar{x})$ and $N_X(\bar{x})$ of smooth manifolds (see \Cref{Prop:SM-basic} (a)), we arrive at the inclusion in \eqref{SumRule-2-Incl2}. 
  
  When every $S_i\vert_X$ is graph-convex and $X$ is regular around $\bar{x}$, by \cite[Theorem 6.42, Theorem 10.41]{VaAn}, the inclusion \eqref{SumRule-NC-Incl1} becomes an equation and the union is superfluous. Thus the inclusions \eqref{SumRule-2-Incl1} and \eqref{SumRule-2-Incl2} become equations accordingly. 
\end{proof}
Comparing the constraint qualification (\ref{CQforSumRule2}) with the one in \cite[Theorem 10.41]{VaAn}, we can see that (\ref{CQforSumRule2}) also serves as a constraint qualification to express $N_{\gph S\vert_X}$ via $N_{\gph S}$ and $N_X$. Next we present a sum rule where each $S_i$ is restricted onto $X$. 
\begin{theorem}[Sum rule-2]\label{SumRule-4}
  Let $S = S_1 + \dots + S_p$ for $S_i : \mathbb{R}^n \rightrightarrows \mathbb{R}^m$ being outer semicontinuous relative to $X$ and let $\bar{x} \in \dom S\cap X$, $\bar{u} \in S\vert_X (\bar{x})$. Assume the boundedness condition (\ref{SumRule-1-BddCond}) is satisfied and the following constraint qualification holds:
      \begin{equation}\label{CQforSumRule4}
      \left. \begin{matrix}
          v_i \in D^*S_i\vert_X (\bar{x} \mid u_i) (0), \ u_i \in S_i (\bar{x}), \ \sum_{i=1}^{p} u_i = \bar{u}\\
          \sum_{i=1}^{p} v_i = 0
              \end{matrix} \right\} \Longrightarrow v_i =0 \text{ for } i = 1, \dots ,p.
      \end{equation}
     Then $\gph S\vert_X$ is locally closed at $(\bar{x}, \bar{u})$ and one has
   \begin{equation}\label{SumRule-4-Incl1}  
    D^*_X S(\bar{x} \mid\bar{u})(y)  \subseteq \limsup_{\footnotesize \substack{ (x,u)\xrightarrow[]{\gph S\vert_X}(\bar{x}, \bar{u})\\  y' \xrightarrow[]{}y}}\   \bigcup_{\substack{ u'_i\in S_i (x) \\  \sum_{i=1}^{p} u'_i =u}}  {\rm proj}_{T_X(x)} \left(\sum_{i=1}^{p} D^*S_i \vert_X (x \mid u'_i)( y')  \right) .
  \end{equation}
  If $X$ is a smooth manifold around $\bar{x}$, 
    \begin{equation}\label{SumRule-4-smincl}  
    D^*_X S(\bar{x} \mid\bar{u})(y)  \subseteq   \bigcup_{\substack{ u_i\in S_i (\bar{x}) \\  \sum_{i=1}^{p} u_i=\bar{u}}}  {\rm proj}_{T_X(\bar{x})} \left(\sum_{i=1}^{p} D^*S_i \vert_X (\bar{x} \mid u_i)( y)  \right) .
  \end{equation}
  When $S_i\vert_X$ are graph-convex, the inclusions \eqref{SumRule-4-Incl1} and \eqref{SumRule-4-smincl} become equations respectively and the union is superfluous. 
\end{theorem}
\begin{proof}
  By \cite[Theorem 10.41]{VaAn}, for any $(x,u) \in \gph S\vert_X$ being close enough to $(\bar{x}, \bar{u})$:
  \begin{equation}\label{sumrule-2-pf-incl1}
    N_{\gph S\vert_X}(x, u) \subseteq  \bigcup_{\substack{ u'_i\in S_i (x) \\  \sum_{i=1}^{p} u'_i =u}} \left\{  \left(\sum_{i=1}^{p} v_i, -y' \right)  \Bigg\vert \  v_i \in D^*S_i \vert_X (x \mid u'_i)(y') \right\}
  \end{equation}
  and therefore
  \begin{equation}\label{sumrule-2-pf-incl2}
      \begin{aligned}
         & \limsup_{\footnotesize (x,u)\xrightarrow[]{\gph S\vert_X}(\bar{x}, \bar{u})}{\rm proj}_{T_X(x) \times \mathbb{R}^m} N_{\gph S\vert_X}(x, u) \\
         \subseteq  & \limsup_{\footnotesize (x,u)\xrightarrow[]{\gph S\vert_X}(\bar{x}, \bar{u})} \bigcup_{\substack{ u'_i\in S_i (x) \\  \sum_{i=1}^{p} u'_i =u}} \left\{    \left( {\rm proj}_{T_X(x)} \left( \sum_{i=1}^{p} v_i \right) , -y' \right) \Bigg\vert \ v_i \in D^*S_i \vert_X (x \mid u'_i)(y') \right\}. 
      \end{aligned}
  \end{equation}
  When $X$ is a smooth manifold around $\bar{x}$, by \Cref{coro:Osc-ProjCode} and \eqref{sumrule-2-pf-incl1}, 
    \begin{align*}
     D^*_X S(\bar{x} \mid\bar{u})(y)  =  & \   {\rm proj}_{T_X(\bar{x})} D^*S\vert_X(\bar{x}\mid\bar{u})(y)  \\
      \subseteq & \bigcup_{\substack{ u_i\in S_i(\bar{x}) \\  \sum_{i=1}^{p} u_i =\bar{u}}}   \left\{   {\rm proj}_{T_X(\bar{x})}\left( \sum_{i=1}^{p} v_i \right) \Bigg\vert \ v_i \in D^*S_i \vert_X(\bar{x} \mid u_i) (y)  \right\}.
  \end{align*}
  When $S_i\vert_X$ are convex, the inclusions \eqref{sumrule-2-pf-incl1} and \eqref{sumrule-2-pf-incl2} become equations and the union is superfluous.
\end{proof}
Next we illustrate why \eqref{CQforSumRule2} is stronger than \eqref{CQforSumRule4}. Without loss of generality we consider a summation of two mappings : $S = S_1 + S_2$. Then the constraint qualification \eqref{CQforSumRule2} writes
\begin{equation}\label{cq:stronger-twosum}
            \left. \begin{matrix}
         u_1 \in S_1 (\bar{x}), \  u_2 \in S_2 (\bar{x}), \  \sum_{i=1}^{2} u_i = \bar{u}\\
          v_1 \in D^*S_1 (\bar{x} \mid u_1) (0) \\
          v_2 \in D^*S_2(\bar{x} \mid u_2)(0) \\
          0 \in v_1 + v_2  + N_X(\bar{x})
    \end{matrix} \right\} \Longrightarrow v_1, v_2 =0. 
\end{equation}
Take $v_2 = 0$, then we have 
$$
     \left. \begin{matrix}
         u_1 \in S_1 (\bar{x})\\
          v_1 \in D^*S_1 (\bar{x} \mid u_1) (0) \\
          0 \in v_1   + N_X(\bar{x})
    \end{matrix} \right\} \Longrightarrow v_1 = 0.
$$
By \cite[Theorem 6.42]{VaAn}, we have $D^*S_1 \vert_X (\bar{x} \mid u_1) (0) \subseteq D^*S_1 (\bar{x} \mid u_1 )(0) + N_X(\bar{x}).$ Similarly we also have $D^*S_2 \vert_X (\bar{x} \mid u_2) (0) \subseteq D^*S_2 (\bar{x} \mid u_2 )(0) + N_X(\bar{x}).$ 
Likewise, \eqref{CQforSumRule2} enables $D^*S_i \vert_X (x\mid u_i) \subseteq D^*S_i (x\mid u_i)  + N_X(x)$ for each $S_i$ 
and thus \eqref{CQforSumRule2} indicates \eqref{CQforSumRule4}. Such a result is intuitive as \eqref{SumRule-2-Incl2} is neater while \eqref{SumRule-4-smincl} still carries $X$ into calculation with $D^*S_i \vert_X(\bar{x}\mid u_i)$. 

Next we provide an example to show that the constraint qualification \eqref{CQforSumRule2} is stronger than \eqref{CQforSumRule4}. In this example, \eqref{CQforSumRule2} is not satisfied and thus we can only turn to \Cref{SumRule-4} for calculation.

\begin{example}
Let $S_1, S_2 : \mathbb{R} \rightrightarrows \mathbb{R}$ be two set-valued mappings defined via $$\gph S_1 = \mathbb{R}^2_- , \  \gph S_2 =\mathbb{R}_+ \times \mathbb{R}_-$$ and $S = S_1 + S_2$. For $(\bar{x},\bar{u} ) = (0, 0) \in \gph S$, the only possible choice of $u_1 \in S_1(\bar{x})$, $u_2 \in S_2(\bar{x})$ with $ u_1 + u_2 =\bar{u}$ is $u_1 =0, u_2= 0$. 
The boundedness condition \eqref{SumRule-1-BddCond} is satisfied at $(\bar{x} , \bar{u} )$. By simple calculation we have 
\begin{equation*}
       N_{\gph S_1}(\bar{x}, u_1) = \mathbb{R}^2_+ , \ 
       N_{\gph S_2}(\bar{x}, u_2) = \mathbb{R}_- \times \mathbb{R}_+, 
\end{equation*}
and accordingly
\begin{equation*}
        D^*S_1 (\bar{x}\mid  u_1) (y)  = \begin{cases}
        \mathbb{R}_+ , & \text{ if } y \in \mathbb{R}_-\\
       \emptyset, & \text{ else} 
    \end{cases} , \ 
    D^*S_2 (\bar{x}\mid  u_2) (y)  = \begin{cases}
        \mathbb{R}_- , & \text{ if } y \in \mathbb{R}_-\\
       \emptyset, & \text{ else}
       \end{cases}.
\end{equation*}
Let $X = \dom S = \{0\}$. Immediately we have $N_X(\bar{x}) = \mathbb{R}$ and $T_X(\bar{x}) = \{0\}$. 
Then we can see that the constraint qualification \eqref{CQforSumRule2} is not satisfied as 
      \begin{equation*}
      \left. \begin{matrix}
        v_1 \in  D^*S_1 (\bar{x}\mid  u_1) (0)=  \mathbb{R}_+\\
          v_2 \in  D^*S_2 (\bar{x}\mid  u_2) (0)= \mathbb{R}_- \\ 
          0\in v_1 + v_2 + \mathbb{R}
              \end{matrix}  \ \right\} \nRightarrow v_1, v_2 = 0.
      \end{equation*}
By restricting $S_1$ and $S_2$ to $X$, we have $\gph S_1 \vert_X = \gph S_2 \vert_X  = \{0\} \times \mathbb{R}_- $ and  
\begin{equation*}
        D^*S_1 \vert_X (\bar{x}\mid  u_1) (y)  = D^*S_2 \vert_X (\bar{x}\mid  u_2) (y) = \mathbb{R}_+, \  \forall y \in \mathbb{R}. 
\end{equation*}
By this representation, the constraint qualification \eqref{CQforSumRule4} is satisfied.  Given that $X$ is also a smooth manifold at $\bar{x}$ and that $\gph S_1 \vert_X$, $\gph S_2 \vert_X$ are convex, both by \Cref{SumRule-4} and by definition we have 
\begin{equation*}
\begin{aligned}
    D^*_X  S (\bar{x}\mid \bar{u}) (y) = D^*_{\dom S} S (\bar{x}\mid \bar{u}) (y) & = {\rm proj}_{T_X(\bar{x})} \left(D^*S_1 \vert_X (\bar{x}\mid  u_1) (y)  + D^*S_2 \vert_X (\bar{x}\mid  u_1) (y) \right) \\
    & ={\rm proj}_{\{0\} } ( \mathbb{R}_+) = \{0\}, \  \forall y \in \mathbb{R}.  
\end{aligned}
\end{equation*}
\end{example}

\section{Conclusions}\label{sect:conclu}
In this article, we focused on introducing more properties of a newly introduced tool, projectional coderivatives, and deriving the corresponding calculus rules. By exploiting the structure of smooth manifolds, we simplified the expression of projectional coderivatives of any set-valued mapping relative to a smooth manifold to a fixed-point one and subsequently extended the generalized Mordukhovich criterion to such a setting. For a closed set in general, we introduced the chain rule of this tool for composition of two set-valued mappings with outer semicontinuity. Along with different constraint qualifications that feature different levels of the set restriction $X$, we introduced two sum rules for users to compute according to their problem structures. An example was provided to illustrate the difference.








\bibliography{References.bib}


\begin{thebibliography}{28}
\ifx \bisbn   \undefined \def \bisbn  #1{ISBN #1}\fi
\ifx \binits  \undefined \def \binits#1{#1}\fi
\ifx \bauthor  \undefined \def \bauthor#1{#1}\fi
\ifx \batitle  \undefined \def \batitle#1{#1}\fi
\ifx \bjtitle  \undefined \def \bjtitle#1{#1}\fi
\ifx \bvolume  \undefined \def \bvolume#1{\textbf{#1}}\fi
\ifx \byear  \undefined \def \byear#1{#1}\fi
\ifx \bissue  \undefined \def \bissue#1{#1}\fi
\ifx \bfpage  \undefined \def \bfpage#1{#1}\fi
\ifx \blpage  \undefined \def \blpage #1{#1}\fi
\ifx \burl  \undefined \def \burl#1{\textsf{#1}}\fi
\ifx \doiurl  \undefined \def \doiurl#1{\url{https://doi.org/#1}}\fi
\ifx \betal  \undefined \def \betal{\textit{et al.}}\fi
\ifx \binstitute  \undefined \def \binstitute#1{#1}\fi
\ifx \binstitutionaled  \undefined \def \binstitutionaled#1{#1}\fi
\ifx \bctitle  \undefined \def \bctitle#1{#1}\fi
\ifx \beditor  \undefined \def \beditor#1{#1}\fi
\ifx \bpublisher  \undefined \def \bpublisher#1{#1}\fi
\ifx \bbtitle  \undefined \def \bbtitle#1{#1}\fi
\ifx \bedition  \undefined \def \bedition#1{#1}\fi
\ifx \bseriesno  \undefined \def \bseriesno#1{#1}\fi
\ifx \blocation  \undefined \def \blocation#1{#1}\fi
\ifx \bsertitle  \undefined \def \bsertitle#1{#1}\fi
\ifx \bsnm \undefined \def \bsnm#1{#1}\fi
\ifx \bsuffix \undefined \def \bsuffix#1{#1}\fi
\ifx \bparticle \undefined \def \bparticle#1{#1}\fi
\ifx \barticle \undefined \def \barticle#1{#1}\fi
\bibcommenthead
\ifx \bconfdate \undefined \def \bconfdate #1{#1}\fi
\ifx \botherref \undefined \def \botherref #1{#1}\fi
\ifx \url \undefined \def \url#1{\textsf{#1}}\fi
\ifx \bchapter \undefined \def \bchapter#1{#1}\fi
\ifx \bbook \undefined \def \bbook#1{#1}\fi
\ifx \bcomment \undefined \def \bcomment#1{#1}\fi
\ifx \oauthor \undefined \def \oauthor#1{#1}\fi
\ifx \citeauthoryear \undefined \def \citeauthoryear#1{#1}\fi
\ifx \endbibitem  \undefined \def \endbibitem {}\fi
\ifx \bconflocation  \undefined \def \bconflocation#1{#1}\fi
\ifx \arxivurl  \undefined \def \arxivurl#1{\textsf{#1}}\fi
\csname PreBibitemsHook\endcsname

\bibitem{aubin1984lipschitz}
\begin{barticle}
\bauthor{\bsnm{Aubin}, \binits{J.-P.}}:
\batitle{Lipschitz behavior of solutions to convex minimization problems}.
\bjtitle{Mathematics of Operations Research}
\bvolume{9}(\bissue{1}),
\bfpage{87}--\blpage{111}
(\byear{1984})
\end{barticle}
\endbibitem

\bibitem{mordukhovich2018variational}
\begin{bbook}
\bauthor{\bsnm{Mordukhovich}, \binits{B.S.}}:
\bbtitle{Variational Analysis and Applications}.
\bpublisher{Springer},
\blocation{Cham}
(\byear{2018})
\end{bbook}
\endbibitem

\bibitem{mordukhovich1992sensitivity}
\begin{barticle}
\bauthor{\bsnm{Mordukhovich}, \binits{B.S.}}:
\batitle{Sensitivity analysis in nonsmooth optimization}.
\bjtitle{Theoretical Aspects of Industrial Design}
\bvolume{58},
\bfpage{32}--\blpage{46}
(\byear{1992})
\end{barticle}
\endbibitem

\bibitem{mordukhovich1994generalized}
\begin{barticle}
\bauthor{\bsnm{Mordukhovich}, \binits{B.S.}}:
\batitle{Generalized differential calculus for nonsmooth and set-valued
  mappings}.
\bjtitle{Journal of Mathematical Analysis and Applications}
\bvolume{183}(\bissue{1}),
\bfpage{250}--\blpage{288}
(\byear{1994})
\end{barticle}
\endbibitem

\bibitem{VaAn}
\begin{bbook}
\bauthor{\bsnm{Rockafellar}, \binits{R.T.}},
\bauthor{\bsnm{Wets}, \binits{R.J.-B.}}:
\bbtitle{Variational Analysis}.
\bpublisher{Springer},
\blocation{Berlin}
(\byear{2009})
\end{bbook}
\endbibitem

\bibitem{mordukhovich2006variational}
\begin{bbook}
\bauthor{\bsnm{Mordukhovich}, \binits{B.S.}}:
\bbtitle{Variational Analysis and Generalized Differentiation I}.
\bpublisher{Springer},
\blocation{Berlin}
(\byear{2006})
\end{bbook}
\endbibitem

\bibitem{mordukhovich2007coderivative}
\begin{barticle}
\bauthor{\bsnm{Mordukhovich}, \binits{B.S.}},
\bauthor{\bsnm{Outrata}, \binits{J.V.}}:
\batitle{Coderivative analysis of quasi-variational inequalities with
  applications to stability and optimization}.
\bjtitle{SIAM Journal on Optimization}
\bvolume{18}(\bissue{2}),
\bfpage{389}--\blpage{412}
(\byear{2007})
\end{barticle}
\endbibitem

\bibitem{LevyMord2004}
\begin{barticle}
\bauthor{\bsnm{Levy}, \binits{A.B.}},
\bauthor{\bsnm{Mordukhovich}, \binits{B.S.}}:
\batitle{Coderivatives in parametric optimization}.
\bjtitle{Mathematical Programming}
\bvolume{99}(\bissue{2}),
\bfpage{311}--\blpage{327}
(\byear{2004})
\end{barticle}
\endbibitem

\bibitem{Huyen2016}
\begin{barticle}
\bauthor{\bsnm{Huyen}, \binits{D.T.K.}},
\bauthor{\bsnm{Yen}, \binits{N.D.}}:
\batitle{Coderivatives and the solution map of a linear constraint system}.
\bjtitle{SIAM Journal on Optimization}
\bvolume{26}(\bissue{2}),
\bfpage{986}--\blpage{1007}
(\byear{2016})
\end{barticle}
\endbibitem

\bibitem{lee2005quadratic}
\begin{bbook}
\bauthor{\bsnm{Lee}, \binits{G.M.}},
\bauthor{\bsnm{Tam}, \binits{N.N.}},
\bauthor{\bsnm{Yen}, \binits{N.D.}}:
\bbtitle{Quadratic Programming and Affine Variational Inequalities: A
  Qualitative Study}.
\bpublisher{Springer},
\blocation{New York}
(\byear{2005})
\end{bbook}
\endbibitem

\bibitem{bonnans2013perturbation}
\begin{bbook}
\bauthor{\bsnm{Bonnans}, \binits{J.F.}},
\bauthor{\bsnm{Shapiro}, \binits{A.}}:
\bbtitle{Perturbation Analysis of Optimization Problems}.
\bpublisher{Springer},
\blocation{Berlin}
(\byear{2013})
\end{bbook}
\endbibitem

\bibitem{dontchev2009implicit}
\begin{bbook}
\bauthor{\bsnm{Dontchev}, \binits{A.L.}},
\bauthor{\bsnm{Rockafellar}, \binits{R.T.}}:
\bbtitle{Implicit Functions and Solution Mappings}.
\bpublisher{Springer},
\blocation{Heidelberg}
(\byear{2009})
\end{bbook}
\endbibitem

\bibitem{ioffe2017variational}
\begin{bbook}
\bauthor{\bsnm{Ioffe}, \binits{A.D.}}:
\bbtitle{Variational Analysis of Regular Mappings: Theory and Applications}.
\bpublisher{Springer},
\blocation{Cham}
(\byear{2017})
\end{bbook}
\endbibitem

\bibitem{klatte2006nonsmooth}
\begin{bbook}
\bauthor{\bsnm{Klatte}, \binits{D.}},
\bauthor{\bsnm{Kummer}, \binits{B.}}:
\bbtitle{Nonsmooth Equations in Optimization: Regularity, Calculus, Methods and
  Applications}.
\bpublisher{Kluwer Academic Publishers},
\blocation{New York}
(\byear{2002})
\end{bbook}
\endbibitem

\bibitem{gfrerer2013directional}
\begin{barticle}
\bauthor{\bsnm{Gfrerer}, \binits{H.}}:
\batitle{On directional metric regularity, subregularity and optimality
  conditions for nonsmooth mathematical programs}.
\bjtitle{Set-Valued and Variational Analysis}
\bvolume{21}(\bissue{2}),
\bfpage{151}--\blpage{176}
(\byear{2013})
\end{barticle}
\endbibitem

\bibitem{ginchev2011directionally}
\begin{barticle}
\bauthor{\bsnm{Ginchev}, \binits{I.}},
\bauthor{\bsnm{Mordukhovich}, \binits{B.S.}}:
\batitle{On directionally dependent subdifferentials}.
\bjtitle{Comptes rendus de l’Acad\'emie bulgare des Sciences}
\bvolume{64}(\bissue{4}),
\bfpage{497}--\blpage{508}
(\byear{2011})
\end{barticle}
\endbibitem

\bibitem{clarke1990optimization}
\begin{bbook}
\bauthor{\bsnm{Clarke}, \binits{F.H.}}:
\bbtitle{Optimization and Nonsmooth Analysis}.
\bpublisher{SIAM},
\blocation{Philadelphia}
(\byear{1990})
\end{bbook}
\endbibitem

\bibitem{GfrOut2016}
\begin{barticle}
\bauthor{\bsnm{Gfrerer}, \binits{H.}},
\bauthor{\bsnm{Outrata}, \binits{J.V.}}:
\batitle{On {L}ipschitzian properties of implicit multifunctions}.
\bjtitle{SIAM Journal on Optimization}
\bvolume{26}(\bissue{4}),
\bfpage{2160}--\blpage{2189}
(\byear{2016})
\end{barticle}
\endbibitem

\bibitem{van2015directional}
\begin{barticle}
\bauthor{\bsnm{Van~Ngai}, \binits{H.}},
\bauthor{\bsnm{Th{\'e}ra}, \binits{M.}}:
\batitle{Directional metric regularity of multifunctions}.
\bjtitle{Mathematics of Operations Research}
\bvolume{40}(\bissue{4}),
\bfpage{969}--\blpage{991}
(\byear{2015})
\end{barticle}
\endbibitem

\bibitem{ioffe2010regularity}
\begin{barticle}
\bauthor{\bsnm{Ioffe}, \binits{A.D.}}:
\batitle{On regularity concepts in variational analysis}.
\bjtitle{Journal of Fixed Point Theory and Applications}
\bvolume{8}(\bissue{2}),
\bfpage{339}--\blpage{363}
(\byear{2010})
\end{barticle}
\endbibitem

\bibitem{arutyunov2006directional}
\begin{barticle}
\bauthor{\bsnm{Arutyunov}, \binits{A.V.}},
\bauthor{\bsnm{Izmailov}, \binits{A.F.}}:
\batitle{Directional stability theorem and directional metric regularity}.
\bjtitle{Mathematics of Operations Research}
\bvolume{31}(\bissue{3}),
\bfpage{526}--\blpage{543}
(\byear{2006})
\end{barticle}
\endbibitem

\bibitem{mordukhovich2004restrictive}
\begin{barticle}
\bauthor{\bsnm{Mordukhovich}, \binits{B.S.}},
\bauthor{\bsnm{Wang}, \binits{B.}}:
\batitle{Restrictive metric regularity and generalized differential calculus in
  {B}anach spaces}.
\bjtitle{International Journal of Mathematics and Mathematical Sciences}
\bvolume{2004}(\bissue{50}),
\bfpage{2653}--\blpage{2680}
(\byear{2004})
\end{barticle}
\endbibitem

\bibitem{Benko2020}
\begin{barticle}
\bauthor{\bsnm{Benko}, \binits{M.}},
\bauthor{\bsnm{Gfrerer}, \binits{H.}},
\bauthor{\bsnm{Outrata}, \binits{J.V.}}:
\batitle{Stability analysis for parameterized variational systems with implicit
  constraints}.
\bjtitle{Set-Valued and Variational Analysis}
\bvolume{28}(\bissue{1}),
\bfpage{167}--\blpage{193}
(\byear{2020})
\end{barticle}
\endbibitem

\bibitem{Meng2020}
\begin{barticle}
\bauthor{\bsnm{Meng}, \binits{K.W.}},
\bauthor{\bsnm{Li}, \binits{M.H.}},
\bauthor{\bsnm{Yao}, \binits{W.F.}},
\bauthor{\bsnm{Yang}, \binits{X.Q.}}:
\batitle{Lipschitz-like property relative to a set and the generalized
  {M}ordukhovich criterion}.
\bjtitle{Mathematical Programming}
\bvolume{189}(\bissue{1}),
\bfpage{455}--\blpage{489}
(\byear{2021})
\end{barticle}
\endbibitem

\bibitem{yao2022relative}
\begin{botherref}
\oauthor{\bsnm{Yao}, \binits{W.F.}},
\oauthor{\bsnm{Yang}, \binits{X.Q.}}:
Relative Lipschitz-like property of parametric systems via projectional
  coderivative.
arXiv
(2022).
\doiurl{10.48550/ARXIV.2210.11335}.
\url{https://arxiv.org/abs/2210.11335}
\end{botherref}
\endbibitem

\bibitem{Daniilidis2011}
\begin{barticle}
\bauthor{\bsnm{Daniilidis}, \binits{A.}},
\bauthor{\bsnm{Pang}, \binits{J.C.}}:
\batitle{Continuity and differentiability of set-valued maps revisited in the
  light of tame geometry}.
\bjtitle{Journal of the London Mathematical Society}
\bvolume{83}(\bissue{3}),
\bfpage{637}--\blpage{658}
(\byear{2011})
\end{barticle}
\endbibitem

\bibitem{lee2013smooth}
\begin{bbook}
\bauthor{\bsnm{Lee}, \binits{J.M.}}:
\bbtitle{Introduction to Smooth Manifolds}.
\bpublisher{Springer},
\blocation{New York}
(\byear{2013})
\end{bbook}
\endbibitem

\bibitem{lang2016introduction}
\begin{bbook}
\bauthor{\bsnm{Lang}, \binits{S.}}:
\bbtitle{Introduction to Linear Algebra (5th Edition)}.
\bpublisher{Wellesley - Cambridge Press},
\blocation{Wellesley}
(\byear{2016})
\end{bbook}
\endbibitem

\end{thebibliography}
\end{document}